\documentclass{amsart}
\usepackage{amssymb,amstext,amsmath,amscd,amsthm,amsfonts,enumerate,graphicx,latexsym,stmaryrd,multicol,bm,xcolor}
\usepackage{hyperref}
\usepackage{cleveref}
\hypersetup{colorlinks=true}
\usepackage[all]{xy}

\newtheorem{thm}{Theorem}[section]
\newtheorem{lem}[thm]{Lemma}
\newtheorem{prop}[thm]{Proposition}
\newtheorem{cor}[thm]{Corollary}
\theoremstyle{definition}
\newtheorem{dfn}[thm]{Definition}
\newtheorem{ques}[thm]{Question}
\newtheorem{rem}[thm]{Remark}
\newtheorem{eg}[thm]{Example}
\newtheorem{nota}[thm]{Notation}
\theoremstyle{remark}
\newtheorem*{ac}{Acknowledgments}

\numberwithin{equation}{thm}
\def\A{\mathcal{A}}
\def\add{\operatorname{add}}

\def\BB{\operatorname{BB}}

\def\dimvec{\underline{\dim}}
\def\C{\mathcal{C}}
\def\codim{\operatorname{codim}}
\def\cok{\operatorname{Cok}}

\def\dim{\operatorname{dim}}

\def\edim{\operatorname{edim}}
\def\Ext{\operatorname{Ext}}
\def\gHom{\operatorname{Hom}^{\mathbb{Z}}}
\def\gmod{\operatorname{mod}^{\mathbb{Z}}\!}
\def\Hom{\operatorname{Hom}}
\def\id{\operatorname{id}}
\def\image{\operatorname{Im}}
\def\ker{\operatorname{Ker}}
\def\L{\Lambda}
\def\le{\leqslant}
\def\ge{\geqslant}
\def\m{\mathfrak{m}}
\def\mod{\operatorname{mod}}
\def\n{\mathfrak{n}}
\def\O{\mathcal{O}}

\def\rad{\operatorname{rad}}
\def\reg{\operatorname{reg}}
\def\Serre{\operatorname{Serre}}
\def\soc{\operatorname{soc}}

\def\T{\mathsf{T}}
\def\thick{\operatorname{thick}}
\def\Tor{\operatorname{Tor}}
\def\X{\mathcal{X}}

\def\Z{\mathbb{Z}}

\def\bfa{\mathbf{a}}
\def\bfb{\mathbf{b}}
\def\bfs{\mathbf{s}}
\def\bft{\mathbf{t}}

\def\bfp{\mathbf{p}}
\def\bfq{\mathbf{q}}
\def\bfu{\mathbf{u}}
\def\bfv{\mathbf{v}}

\def\bfz{\mathbf{z}}
\begin{document}
\allowdisplaybreaks

\title{Thick subcategories over weakly symmetric algebras with radical cube zero}
\author{Kaito Kimura}
\address{Department of Mathematics, Purdue University, 150 N. University Street, West Lafayette, IN 47907, USA}
\email{m21018b@gmail.com}

\begin{abstract} 
In this paper, we study thick subcategories of the (stable) module categories of finite-dimensional weakly symmetric algebras with radical cube zero.
When the spectral radius of the Ext matrix is two, thick subcategories containing the algebra, equivalently thick subcategories of the stable category, are classified by Serre subcategories of a certain abelian category; when it is not two, only the trivial thick subcategories occur.
As an application, this yields, to the best of our knowledge, the first examples of commutative Noetherian Gorenstein dominant local rings that are not complete intersections.
\end{abstract}

\maketitle

\section{Introduction}

Let $k$ be an algebraically closed field and let $\L$ be an indecomposable finite-dimensional weakly symmetric $k$-algebra with radical cube zero.
For a complete set of representatives $S_1,\ldots,S_r$ of the isomorphism classes of simple modules, the real symmetric matrix $E$ defined by $E_{ij}=\dim_k\Ext_\L^1(S_i,S_j)$ is called the Ext matrix, and we denote its spectral radius by $\lambda$.
Benson \cite{Ben} showed that the growth of minimal projective resolutions of modules is governed by $\lambda$, and in particular identified two as the critical value.
More precisely, the dimensions of the terms in a minimal projective resolution of a nonprojective module are bounded when $\lambda<2$, are either bounded or grow linearly when $\lambda=2$, and grow exponentially when $\lambda>2$.
This trichotomy also corresponds to the representation type of algebras in this class: the cases $\lambda<2$, $\lambda=2$, and $\lambda>2$ correspond to finite, tame, and wild representation type, respectively.
Related phenomena arising from this spectral trichotomy have also been studied from several viewpoints; see, for example, \cite{ESc,ESo2,ESo}.
Erdmann \cite{E}, in the same setting, studied the existence of Ext-finite nonprojective modules, that is, nonprojective modules $M$ satisfying $\Ext_\L^{\gg 0}(M,M)=0$, and proved that such a module can exist only when $\lambda=2$.

The main results of this paper describe the distinction among the cases $\lambda<2$, $\lambda=2$, and $\lambda>2$ from the viewpoint of the structure of thick subcategories of the category $\mod\L$ of finitely generated left $\L$-modules and of its stable category $\underline{\mod}\L$.
A subcategory is said to be \textit{thick} if, in the abelian case, it is closed under direct summands, kernels of epimorphisms, cokernels of monomorphisms, and extensions, and if, in the triangulated case, it is a full triangulated subcategory closed under direct summands.
When $\lambda\ne2$, there are no nontrivial thick subcategories containing $\L$.

\begin{thm}\label{thm1.1}
If $\lambda\ne 2$, then the only thick subcategories of $\mod\L$ containing $\L$ are $\add\L$ and $\mod\L$. In particular, the stable module category $\underline{\mod}\L$ has no nonzero proper thick subcategory.
\end{thm}

\noindent The proof of this result is given in Section 4, and the theorem generalizes the aforementioned result of Erdmann \cite{E}.
Indeed, higher Ext-orthogonality with respect to a fixed module defines a thick subcategory, and the triviality of such thick subcategories forces the module to be projective; see Corollary \ref{Ext Tor rigidity}.

The situation is quite different when $\lambda=2$. 
For instance, the class of all modules whose minimal projective resolutions are bounded clearly forms a nontrivial thick subcategory.
This naturally raises the question of what the nontrivial thick subcategories are.
Our next main result shows that the structure of minimal projective resolutions, which already plays an important role in \cite{Ben}, is closely related to the structure of thick subcategories.
Since $\L$ has radical cube zero, both $\L$ itself and modules with bounded minimal projective resolutions admit natural gradings.
This gives rise to an abelian category $\A$ consisting of modules with bounded minimal projective resolutions, and the self-injectivity of $\L$ implies that the syzygy functor over $\L$ induces an exact autoequivalence $\Omega$ of $\A$.
The details of these constructions are given successively in Sections 3, 4, and 5.
With this notation, the classification in the case $\lambda=2$ can be stated explicitly as follows.

\begin{thm}\label{thm1.2}
If $\lambda=2$, then there are order-preserving bijections
$$
\left\{
\begin{matrix}
\text{Thick subcategories of}\\
\text{$\mod\L$ containing $\L$}
\end{matrix}
\right\}
\longleftrightarrow
\left\{
\begin{matrix}
\text{$\Omega$-stable Serre}\\
\text{subcategories of $\A$} 
\end{matrix}
\right\}\sqcup\{\gmod\L\}
\longleftrightarrow
\left\{
\begin{matrix}
\text{Thick subcategories}\\
\text{of $\underline{\mod}\L$}
\end{matrix}
\right\}.
$$
\end{thm}

\noindent The proof of this result is given in Section 5. The category $\A$ arising in our classification can be identified with the category of regular modules over a hereditary algebra of Euclidean type; see Remark \ref{hereditary}. Thus, the classical theory of Euclidean regular representations provides a model for the thick subcategory structure of weakly symmetric algebras at the critical spectral radius $\lambda=2$.
Together with the blockwise extension in Section 6, Theorems \ref{thm1.1} and \ref{thm1.2} give a complete description of the thick subcategories of finite-dimensional weakly symmetric $k$-algebras with radical cube zero without assuming indecomposability; see Theorem \ref{general classification}.
We also note that, in the above results, it is not necessary to assume that $k$ is algebraically closed; it suffices to assume that $\L$ is split.

An important application of Theorem \ref{thm1.1} lies in commutative algebra.
Takahashi \cite{T23} introduced \textit{dominance} as a property of commutative Noetherian local rings, providing a framework in which thick subcategories of the category of finitely generated modules can be classified in terms of specialization-closed subsets of the spectrum.
A commutative Noetherian local ring $(R,\m,k)$ is said to be \textit{dominant} if, for every finitely generated $R$-module $M$ of infinite projective dimension, the smallest thick subcategory containing both $R$ and $M$ contains $k$.
In particular, when $R$ is Artinian, $R$ is dominant if and only if the only thick subcategories containing $R$ are $\add R$ and $\mod R$.
Hypersurface local rings are typical examples of dominant local rings, and many other classes of dominant rings are now known; see \cite{KT, T23}.
However, these classes are either incompatible with the Gorenstein property or have the property that their Gorenstein members are necessarily hypersurfaces.
Consequently, no dominant Gorenstein local ring that is not a hypersurface had previously been known.
This led to open questions concerning dominance of Gorenstein local rings in \cite[Question 9.6]{T23} and \cite[Question 7.1]{KT}.

\begin{ques}\label{ques dom Gor intro}
\begin{enumerate}[\rm (1)]
\item Does there exist a dominant Gorenstein local ring that is not a hypersurface?
\item Is a Gorenstein local ring with radical cube zero and embedding dimension greater than two dominant?
\end{enumerate}
\end{ques}

\noindent There are two opposite directions in which hypersurfaces can be generalized, namely complete intersections and Golod rings.
It is known that a dominant complete intersection local ring must be a hypersurface.
On the other hand, Golod rings are conjectured to be dominant.
From this perspective, whether dominant rings exist among Gorenstein rings, which generalize complete intersections, has been one of the most important questions for understanding the ubiquity of dominance.
Theorem \ref{thm1.1}, however, gives an affirmative answer to Question \ref{ques dom Gor intro}(2), and hence also to Question \ref{ques dom Gor intro}(1), in the case of rings containing a field.

\begin{cor}\label{cor1.4}
Every commutative finite-dimensional Gorenstein local $k$-algebra with radical cube zero and embedding dimension greater than two is dominant.
\end{cor}

\noindent Indeed, the embedding dimension corresponds to the spectral radius $\lambda$.
When the embedding dimension is less than two, the ring is either a field or a hypersurface and hence dominant, whereas when the embedding dimension is two, it is a non-hypersurface complete intersection local ring and is already known not to be dominant.
Moreover, a Gorenstein local ring with $\m^2=0$ is a hypersurface, while non-dominant Gorenstein rings with $\m^4=0$ are known to exist.
Thus, the case $\m^3=0$ was the remaining boundary case between these two behaviors.
Corollary \ref{cor1.4} provides further evidence for the rich ubiquity of dominance as a property of local rings.

The organization of this paper is as follows.
Sections 2 and 3 are devoted solely to preliminaries for the main results. 
In Section 2, we describe the special behavior of graded modules concentrated in two degrees, while in Section 3, after equipping $\L$ with a suitable grading, we formulate an autoequivalence induced by the syzygy functor. 
In Sections 4 and 5, we treat the cases $\lambda\ne 2$ and $\lambda=2$, respectively, and prove Theorems \ref{thm1.1} and \ref{thm1.2} using the lemmas established in Sections 2 and 3. 
Section 6 is devoted to stating the corresponding result in the decomposable case. 
Finally, in Section 7, we discuss applications to commutative rings and compare our results with previous work.

\section{Graded modules and thick subcategories}

In this section, we prepare lemmas that form the foundation of the proofs of the main results of this paper. In particular, we consider special properties of graded modules concentrated in degrees zero and one. Throughout this section, let $k$ be a field and $A=\bigoplus_{i\ge0}A_i$ a positively graded, not necessarily commutative, left Noetherian $k$-algebra such that $A_0$ is a finite-dimensional semisimple $k$-algebra. 

\begin{dfn}
(1) We denote by $\gmod A$ the category of finitely generated $\Z$-graded left $A$-modules.
A morphism $\varphi:M\to N$ in $\gmod A$ is a homomorphism of left $A$-modules satisfying $\varphi(M_i)\subseteq N_i$ for all $i\in\Z$, and we denote the set of such morphisms by $\gHom_A(M,N)$.
For a graded left $A$-module $N$ and integers $i,n$, we use the convention $N(i)_n=N_{n+i}$.
The \textit{underlying $A$-module} of a graded $A$-module $M$ means the left $A$-module obtained from $M$ by forgetting the grading.

(2) For $M,N\in\gmod A$, we have 
$$
\Hom_A(M,N)=\bigoplus_{i\in\Z}\gHom_A(M,N(i)), \quad
\Hom_A(M,N)_i=\gHom_A(M,N(i)).
$$
Thus, $\Hom_A(M,N)$ is a graded $k$-vector space, and $\Hom_A(M,N)_0=\gHom_A(M,N)$.
The category $\gmod A$ has enough projective objects, and every finitely generated graded $A$-module has a graded projective resolution in $\gmod A$.
The $j$th Ext module $\Ext_A^j(M,N)$ admits the structure of a graded $k$-vector space induced by a graded projective resolution of $M$.
\end{dfn}

The following lemma is a consequence of the more general correspondence between short exact sequences and elements of first Ext modules. For the reader's convenience, we include a proof only of the assertion that will be used later.

\begin{lem}\label{graded Ext1}
Let $L,M,N\in\gmod A$ with $\Ext_A^1(N,L)_0=0$. Then every short exact sequence $0\to L\to M\to N\to0$ in $\gmod A$ splits.
\end{lem}

\begin{proof}
An exact sequence $0\to L\to M\to N\to0$ in $\gmod A$ induces an exact sequence 
$$
\Hom_A(N,M)\to \Hom_A(N,N)\to\Ext_A^1(N,L)
$$
of graded $k$-vector spaces. Taking the $0$th component of this sequence, $\gHom_A(N,M)\to \gHom_A(N,N)$ is surjective since $\Ext_A^1(N,L)_0=0$, which implies that $M\to N$ splits in $\gmod A$.
\end{proof}

The following lemma gives a sufficient condition for the Ext vanishing assumed in Lemma \ref{graded Ext1}.

\begin{lem}\label{graded Ext2}
Let $L,N\in\gmod A$ and $a=\inf\{i\in\Z\mid N_i\ne0\}$.
Suppose that there is $b\in\Z$ such that $L_i=0$ for all $i\ge b$.
Then, for all $p\ge0$ and $q\in\Z$, one has $\Ext_A^p(N,L)_q=0$ whenever $p+q\ge b-a$.
\end{lem}

\begin{proof}
We may assume that $N\ne 0$, and thus $a\in\Z$.
Since $A_0$ is semisimple, a finitely generated $A_0$-module $N_a$ is projective, and hence $A\otimes_{A_0}N_a$ is a graded projective $A$-module.
The natural map $A\otimes_{A_0}N_a\to N: a\otimes x\mapsto ax$ is an isomorphism in degree $a$. Combining this with a surjection $P\to \bigoplus_{i\ge a+1}N_i$ in $\gmod A$ where $P$ is a graded projective $A$-module with $P_i=0$ for all $i<a+1$, we obtain a surjection $f: F\to N$ such that $F$ is a graded projective $A$-module and  $(\ker f)_i=0$ for every $i<a+1$. 
Repeating this argument, we obtain a graded projective resolution $(\cdots\to F^2\to F^1\to F^0\to0)$ of $N$ such that $F^p_i=0$ for every $i<a+p$.

Let $p\ge0$ and $q\in\Z$ be integers with $p+q\ge b-a$.
For $i\in\Z$, if $i<a+p$, then $F^p_i=0$ by the previous paragraph; otherwise, one has $i+q\ge i+(b-a-p)\ge b$ and hence $L_{i+q}=0$ by assumption.
This means that $\Hom_A(F^p,L)_q=\gHom_A(F^p,L(q))=0$, and thus $\Ext_A^p(N,L)_q=0$.
\end{proof}

The proposition to be proved next, Proposition \ref{kernel cokernel}, will play an important role in the proof of the main result. To state it, we first recall the notion of a thick subcategory.

\begin{dfn}
For a left Noetherian ring $\L$, which is not necessarily graded, we denote by $\mod\L$ the category of finitely generated left $\L$-modules.
Let $\A$ be an abelian category.
We say that a subcategory $\X$ of $\A$ is {\em thick} if it satisfies the following conditions:
\begin{enumerate}[{\rm (1)}]
\item
Let $A$ be an object in $\A$ and $B$ a direct summand of $A$.
If $A$ belongs to $\X$, then $B$ also belongs to $\X$.
\item
Let $0\to A\to B\to C\to0$ be a short exact sequence in $\A$.
If two of the three objects $A,B,C$ belong to $\X$, then the third object also belongs to $\X$.
\end{enumerate}
The \textit{thick closure} of a subcategory $\X$ of $\A$ is defined to be the smallest thick subcategory containing $\X$, and denoted $\thick_\A (\X)$.
\end{dfn}

By definition, a thick subcategory is closed under kernels of epimorphisms and cokernels of monomorphisms between its objects.
The following proposition shows that, for graded modules concentrated in degrees zero and one, the underlying modules of the kernels and cokernels of arbitrary degree-preserving morphisms also belong to the corresponding thick closure.

\begin{prop}\label{kernel cokernel}
Let $M,N\in\gmod A$ with $M_i=0=N_i$ for all $i\in\Z\setminus\{0,1\}$.
For every morphism $f:M\to N$ in $\gmod A$, the underlying $A$-modules of $\ker f$ and $\cok f$ belong to $\thick_{\mod A}\{A,M,N\}$.
\end{prop}

\begin{proof}
Let $p:N\to C$ be the natural surjection, and put $K=\ker f$ and $C=\cok f$. 
We take a surjection $g: F\to N$ in $\gmod A$, where $F$ is a graded free $A$-module.
The pullback diagram 
$$
\xymatrix{
0\ar[r]&K\ar[r]\ar@{=}[d]&
G\ar[r]\ar[d]\ar@{}[dr]|{\text{\scriptsize PB}}
&F\ar[r]^{pg}\ar@{->>}[d]^{g}&
C\ar[r]\ar@{=}[d]&0\\
0\ar[r]&K\ar[r]&
M\ar[r]^{f}&N\ar[r]^{p}&C\ar[r]&0 
}
$$
induces the exact sequences
$$
0\to G\to M\oplus F \xrightarrow{\varphi} N \to 0, \quad 0\to K\to G \to \ker(pg)\to 0
$$
in $\gmod A$, where $\varphi(x,y)=f(x)-g(y)$ for $x\in M$ and $y\in F$.
The first exact sequence implies that the underlying $A$-module of $G$ belongs to $\thick_{\mod A}\{A,M,N\}$.
Both $K$ and $C$ are concentrated in degrees zero and one. 
Lemma \ref{graded Ext2} gives $\Ext_A^1(\ker(pg),K)_0\cong\Ext_A^2(C,K)_0=0$. 
Lemma \ref{graded Ext1} shows that the exact sequence $0\to K\to G \to \ker(pg)\to 0$ splits in $\gmod A$.
Hence the direct summands $K$ and $\ker(pg)$ of $G$ belong to $\thick_{\mod A}\{A,M,N\}$, and hence so does $C$.
\end{proof}

\section{Graded syzygies and cosyzygies}

This section is devoted to collecting the properties of weakly symmetric algebras with radical cube zero that will play a central role in this paper.
The most important result prepared in this section is a lemma showing that the syzygy and cosyzygy functors induce mutually inverse autoequivalences on an abelian category associated with a fixed thick subcategory.
We first recall some terminology.

\begin{dfn}
Let $k$ be a field and let $\L$ be a finite-dimensional $k$-algebra. 
The algebra $\L$ is called \textit{split} if $\L/\rad\L\cong\prod_{i=1}^r M_{n_i}(k)$ as $k$-algebras for some positive integers $r,n_1,\ldots,n_r$. 
(Notice that $\L$ is automatically split when $k$ is algebraically closed.)
We say that $\L$ has \textit{radical cube zero} if $(\rad\L)^3=0$. 
Also, $\L$ is called \textit{self-injective} if $\!{}_{\L}\L$ is injective, and a self-injective algebra $\L$ is called \textit{weakly symmetric} if the socle and the top of every indecomposable projective left $\L$-module are isomorphic.
\end{dfn}

We first dispose of the case where the radical has square zero.

\begin{rem}\label{square zero}
Let $k$ be a field and let $\L$ be a finite-dimensional split indecomposable weakly symmetric $k$-algebra. 
If $(\rad\L)^2=0$, then the only thick subcategories of $\mod\L$ containing $\L$ are $\add\L$ and $\mod\L$. Indeed, since the assertion is invariant under Morita equivalence, we may assume that $\L$ is basic.
For every indecomposable projective module $P$, the condition $(\rad\L)^2=0$ gives $\rad P\subseteq\soc P$, while weak symmetry says that $\soc P$ is simple and isomorphic to $\operatorname{top} P$. There are no nonzero homomorphisms between indecomposable projective modules with nonisomorphic tops. Since $\L$ is indecomposable, it follows that $\L$ has only one isomorphism class of simple modules, and hence is local. It follows that it is either $k$ or $k[x]/(x^2)$.
\end{rem}

Therefore, the case of interest to us is when $(\rad\L)^2\ne0$.
To avoid repeatedly restating definitions, we fix below the notation used throughout this section.

\begin{nota}\label{notation}
Throughout the rest of this section, let $k$ be a field and let $\L$ be a finite-dimensional split indecomposable weakly symmetric $k$-algebra with radical cube zero, and assume that $J^2\ne0$, where $J=\rad\L$. Choose a semisimple subalgebra $\L_0$ of $\L$ such that $\L=\L_0\oplus J$ and $\L_0\cong\L/J$ by the Wedderburn--Malcev theorem, see \cite[Theorem 72.19]{CR} or \cite[Theorem 11.6]{P} for example.  Write $\L_0=\prod_{i=1}^r M_{n_i}(k)$, and choose a primitive idempotent $e_i\in M_{n_i}(k)$ for each $i$.
Following \cite{Ben}, we set
\begin{align*}
P_i&:=\L e_i, \quad S_i:=P_i/JP_i, \quad E_{ij}:=\dim_k e_j(J/J^2)e_i\ (=\dim_k\Ext_\L^1(S_i,S_j)),
\\
&E:=(E_{ij})_{1\le i,j\le r}, \quad
\lambda:=\max\{|\mu|\mid \text{$\mu$ is an eigenvalue of $E$}\}.
\end{align*}
Then $P_1,\ldots,P_r$ and $S_1,\ldots,S_r$ form complete sets of representatives of the isomorphism classes of indecomposable projective and simple left $\L$-modules, respectively; see \cite[Theorem 25.3]{Lam}.
The number $\lambda$ is called the \textit{spectral radius} of $E$.
\end{nota}

The following two lemmas collect some standard facts about weakly symmetric algebras with radical cube zero; see \cite{Ben, E, ESo, Lam} for example.
In view of the subtle differences in the hypotheses and the fact that the results of this paper have important applications to commutative algebra, we include the proof for the reader's convenience.

\begin{lem}\label{basic properties1}
The following statements hold.
\begin{enumerate}[{\rm (1)}]
\item For all $1\le i,j\le r$, the equality $\dim_k e_i\L e_j=\dim_k e_j\L e_i$ holds.
\item One has $J^2=\soc({}_\L\L)$. More precisely, $J^2e_i=\soc(P_i)\cong S_i$ and $\dim_k e_jJ^2e_i=\delta_{ij}$ for all $1\le i,j\le r$.
\item The matrix $E$ is symmetric.
\item There is no nonempty proper subset $I\subset\{1,\ldots,r\}$ such that $E_{ij}=0$ for every $i\in I$ and $j\notin I$.
\end{enumerate}
\end{lem}

\begin{proof}
(1) Since $\L$ is self-injective, $P_j$ is injective, and weak symmetry gives $\soc(P_j)\cong S_j$. Hence $P_j$ is the injective envelope of $S_j$.
Set $D(-)=\Hom_k(-,k)$. As $e_j\L$ is a projective right $\L$-module, $D(e_j\L)$ is an injective left $\L$-module, and $\soc D(e_j\L)\cong D(e_j\L/e_jJ)\cong S_j$.
Thus $D(e_j\L)$ is also the injective envelope of $S_j$, and hence $P_j\cong D(e_j\L)$.
So $\dim_k e_i\L e_j=\dim_k e_iD(e_j\L)=\dim_k D(e_j\L e_i)=\dim_k e_j\L e_i$.

(2) Since $J^3=0$, one has $J^2e_i\subseteq\soc(P_i)\cong S_i$ for every $i$, where the isomorphism follows from the weak symmetry of $\L$. Hence either $J^2e_i=0$ or $J^2e_i=\soc(P_i)$. 
It suffices to show that the latter holds for every $1\le i\le r$.

Suppose that $J^2e_i=0$ for some $i$. Then $Je_i\subseteq\soc(P_i)\cong S_i$, so $e_jJe_i=0$ for every $j\ne i$. 
Since $e_i$ and $e_j$ belong to distinct simple components of $\L_0$, $e_j\L e_i/e_jJe_i\cong e_jS_i=0$, and thus $e_j\L e_i=0$ for every $j\ne i$. 
By (1), $e_i\L e_j=0$ for every $j\ne i$. 
Choose a complete set of primitive orthogonal idempotents $f_1,\ldots,f_m$ in $\L_0$, and let $e$ be the sum of those $f_h$ for which $\L f_h\cong P_i$. Since primitive idempotents in the same simple component of $\L_0$ are conjugate by units of $\L_0$, the preceding vanishing implies that $f_h\L f_\ell=0=f_\ell\L f_h$ whenever $\L f_h\cong P_i$ and $\L f_\ell\not\cong P_i$. Hence $e\L(1-e)=0=(1-e)\L e$. Since $ea=eae=ae$ for every $a\in\L$, $e$ is a central idempotent, and thus indecomposability gives $e=1$. Hence every indecomposable projective module is isomorphic to $P_i$, and $J^2e_i=0$ forces $J^2=0$, a contradiction.

(3) Let $i\ne j$. Then $e_j\L e_i/e_jJe_i\cong e_jS_i=0$, and thus $e_j\L e_i=e_jJe_i$. By (2), $J^2e_i\cong S_i$, so $e_jJ^2e_i=0$. Therefore we get
$$
E_{ij}=\dim_k e_j(J/J^2)e_i=\dim_k e_jJe_i=\dim_k e_j\L e_i.
$$
Similarly, $E_{ji}=\dim_k e_i\L e_j$. The assertion (1) gives $E_{ij}=E_{ji}$.

(4) Suppose that there is a nonempty proper subset $I\subset\{1,\ldots,r\}$ such that $E_{ij}=0$ for every $i\in I$ and $j\notin I$. Fix $i\in I$ and $j\notin I$. By (2), one has $e_jJ^2e_i\cong e_jS_i=0$. It follows that
$$
e_jJe_i=e_jJe_i/e_jJ^2e_i\cong e_j(J/J^2)e_i=E_{ij}=0, \ \text{and thus}\ 
e_j\L e_i=e_j\L e_i/e_jJ e_i\cong e_jS_i=0
$$
since $i\ne j$.
By (3), $E_{ji}=0$, and the same argument gives $e_i\L e_j=0$. 
Choose a complete set of primitive orthogonal idempotents $f_1,\ldots,f_m$ in $\L_0$, and let $e$ be the sum of those $f_h$ for which $\L f_h\cong P_i$ for some $i\in I$. By the same reasoning as in the proof of (2), the vanishing of  $e_i\L e_j$ and $e_j\L e_i$ shows that
$$
e\L(1-e)=0=(1-e)\L e.
$$
Hence $ea=eae=ae$ for every $a\in\L$, so $e$ is a central idempotent.
Since $I$ is nonempty and proper, $e\ne0,1$, contradicting the
indecomposability of $\L$.
\end{proof}

\begin{lem}\label{basic properties2}
Let $M\in\mod\L$. The following statements hold.
\begin{enumerate}[{\rm (1)}]
\item There exists a grading $\L=\L_0\oplus\L_1\oplus\L_2$ such that $J=\L_1\oplus\L_2$, $\L_2=J^2$, and $\dim_k e_i\L_1e_j=E_{ij}$ for all $1\le i,j\le r$.
In particular, for all $1\le i\le r$, $P_i$ admits a grading $P_i=\L_0 e_i\oplus\L_1 e_i\oplus\L_2 e_i$ as a graded left $\L$-module.
\item Assume that $J^2M=0$.
With respect to any such grading of $\L$ as in {\rm (1)}, $M$ admits a grading $M=M_0\oplus M_1$ as a graded left $\L$-module such that $\rad M=M_1$.
\item If $M$ has no nonzero projective direct summand, then $J^2M=0$.
\item If $M$ satisfies $J^2M=0$ and has no simple direct summand, then $\soc M=\rad M$.
\end{enumerate}
\end{lem}

\begin{proof}
(1) Since $\L_0$ is split semisimple, $\L_0\otimes_k\L_0^{\rm op}$ is semisimple. Hence the exact sequence
$$
0\to J^2\to J\to J/J^2\to0
$$
splits as a sequence of $\L_0$-bimodules. Choose an $\L_0$-subbimodule $\L_1$ of $J$ such that $J=\L_1\oplus J^2$, and set $\L_2=J^2$. Then $\L=\L_0\oplus\L_1\oplus\L_2$.
Moreover, $\L_1\L_1\subseteq J^2=\L_2$, while $\L_1\L_2=\L_2\L_1=\L_2^2=0$ as $J^3=0$. This decomposition defines a grading of $\L$. Finally, we have $\dim_ke_i\L_1e_j=\dim_ke_i(J/J^2)e_j=E_{ji}=E_{ij}$ by Lemma \ref{basic properties1}(3).

(2) Since $\L_0$ is semisimple, we can choose a left $\L_0$-submodule $M_0$ of $M$ complementary to $\rad M$, and set $M_1=\rad M$. Then $M=M_0\oplus M_1$ as a left $\L_0$-module. Since $J^2M=0$, one has $\L_1M_0\subseteq\rad M=M_1$, $\L_1M_1=0$, and $\L_2M=0$. Hence $M=M_0\oplus M_1$ is a graded left $\L$-module.

(3) Suppose that $J^2M\ne0$. Choose $x\in M$ and a primitive idempotent $e$ in a complete primitive decomposition of $1$ in $\L_0$ such that $J^2ex\ne0$, and set $y=ex$ and $P=\L e$. The homomorphism $f:P\to M$ given by $ae\mapsto ay$ is nonzero on the simple socle $J^2e=\soc(P)$; see Lemma \ref{basic properties1}(2). Since this socle is essential in the indecomposable injective module $P$, one has $\ker f=0$. As $P$ is injective, $f$ splits, and hence $M$ has a nonzero projective direct summand.

(4) Assume that $J^2M=0$. Then $\rad M=JM\subseteq\soc M$. If the inclusion is strict, then there is a simple submodule $S$ of $\soc M$ not contained in $\rad M$. 
The composition $S\to M\to M/\rad M$ of the natural morphisms is nonzero.
As $S$ is simple, this composition is a monomorphism of $\L/J$-modules and hence splits. 
Therefore $S\to M$ also splits.
\end{proof}

\begin{rem}\label{unique grading}
The gradings in Lemma \ref{basic properties2}(1) and (2) depend on the choice of complements, but are unique up to graded isomorphism. Indeed, suppose that
$$
\L=\L_0\oplus\L_1\oplus J^2 \quad\text{and}\quad \L=\L_0\oplus\L'_1\oplus J^2
$$
are two such gradings of $\L$. 
The natural maps $\L_1\to J/J^2$ and $\L'_1\to J/J^2$ are isomorphisms of $\L_0$-bimodules, and hence induce an $\L_0$-bimodule isomorphism $\theta:\L_1\to\L'_1$ such that $\theta(x)-x\in J^2$ for every $x\in\L_1$.  Extending $\theta$ by the identity on $\L_0\oplus J^2$ gives a graded algebra isomorphism between the two gradings: for $x,y\in\L_1$, one has $\theta(x)\theta(y)=xy$ since $J^3=0$.
Similarly, in Lemma \ref{basic properties2}(2), once a grading of $\L$ is fixed, any two choices of $\L_0$-module complements of $\rad M$ in $M$ are isomorphic via the identifications with $M/\rad M$, and this isomorphism extends to an isomorphism in $\gmod\L$ between the corresponding gradings of $M$ since $J^2M=0$.
\end{rem}

To exploit the self-injectivity of $\L$ in the graded setting, we first recall the graded $k$-dual.

\begin{dfn}
Fix a grading of $\L$ as in Lemma \ref{basic properties2}(1).
For a finitely generated graded left (resp.\! right) $\L$-module $N$, its \textit{graded $k$-dual} is
$$
D(N)=\Hom_k(N,k),
\quad
D(N)_i=\Hom_k(N_{-i},k).
$$
We regard $D(N)$ as a graded right (resp.\! left) $\L$-module via $(\varphi a)(x)=\varphi(ax)$ (\text{resp.} $(a\varphi)(x)=\varphi(xa)$) for $a\in\L$, $\varphi\in D(N)$, and $x\in N$. 
In particular, the graded $k$-dual gives a duality between finitely generated graded left and right $\L$-modules.
\end{dfn}

The next lemma records the graded duality needed for cosyzygies.

\begin{lem}\label{Matlis duality}
Fix a grading of $\L$ as in Lemma \ref{basic properties2}(1), and let $D(-)=\Hom_k(-,k)$ be the graded $k$-dual. Then $D(\L_\L)\cong {}_\L\L(2)$ and $D({}_\L\L)\cong \L_\L(2)$ as graded left and graded right $\L$-modules, respectively.
\end{lem}

\begin{proof}
By weak symmetry and \cite[(3.15) and the discussion following Corollary 16.64]{Lam2}, $\L$ is Frobenius and there is a $k$-linear map $\psi:\L\to k$ whose kernel contains no nonzero left ideal. 
Define $\varphi:\L\to k$ by $\varphi|_{\L_2}=\psi|_{\L_2}$ and
$\varphi|_{\L_i}=0$ for $i\ne2$.
If a nonzero left ideal $I$ were contained in $\ker\varphi$, then the nonzero  left ideal $I\cap\soc({}_\L\L)\subseteq\L_2$ is contained in $\ker\psi$, a contradiction.
Thus $\L$ is $2$-Frobenius in the sense of \cite[Definition 5.3]{BT}, and \cite[Lemma 5.5]{BT} gives the assertion. 
\end{proof}

We consider an abelian subcategory of the category of graded modules that reflects the structure of the thick subcategory.

\begin{lem}\label{thick to abelian}
Let $\X$ be a thick subcategory of $\mod\L$ containing $\L$.
The following statements hold.
\begin{enumerate}[{\rm (1)}]
\item If $\X$ contains a simple $\L$-module, then $\X=\mod\L$.
\item Fix a grading of $\L$ as in Lemma \ref{basic properties2}(1). The following hold for the subcategory defined as follows:
$$
\A_\X=\{M\in\gmod\L\mid M_i=0\text{ for }i\notin\{0,1\}\text{ and the underlying $\L$-module of $M$ belongs to $\X$}\}.
$$
\begin{enumerate}[{\rm (i)}]
\item $\A_\X$ is an abelian category of finite length and is closed under extensions in $\gmod\L$.
\item If $\X$ contains a nonprojective module, then $\A_\X$ contains a nonzero object.
\item If $\X\ne\mod\L$ and $0\ne M\in\A_\X$, then one has $M_0\ne0$ and $M_1=\rad M=\soc M=JM_0\ne0$, and the underlying $\L$-module of $M$ has no simple direct summand.
\end{enumerate}
\end{enumerate}
\end{lem}

\begin{proof}
(1) Suppose that $I=\{i\mid S_i\in\X\}\subseteq\{1,\ldots,r\}$ is nonempty. If $i\in I$, then the exact sequences
$$
0\to JP_i\to P_i\to S_i\to0,
\quad
0\to J^2P_i\to JP_i\to JP_i/J^2P_i\to0
$$
show that $JP_i/J^2P_i\in\X$, since $J^2P_i\cong S_i$ by Lemma \ref{basic properties1}(2). As
$$
JP_i/J^2P_i\cong\bigoplus_{j=1}^r S_j^{\oplus E_{ij}}, 
$$ 
it follows that $S_j\in\X$ whenever $E_{ij}>0$.
Thus $I$ satisfies the condition that $E_{ij}=0$ for every $i\in I$ and $j\notin I$.
Lemma \ref{basic properties1}(4) forces $I=\{1,\ldots,r\}$.
Since every finitely generated $\L$-module has finite length, the two-out-of-three property then gives $\X=\mod\L$.

(2)(i) Keeping track of the range of degrees, the fact that $\X$ is thick implies that $\A_\X$ is also a thick subcategory of $\gmod\L$.
Let $f:M\to N$ be a morphism in $\A_\X$.
Proposition \ref{kernel cokernel} shows that the underlying $\L$-modules of $\ker f$ and $\cok f$ belong to $\thick_{\mod\L}\{\L,M,N\}\subseteq\X$.
Since both are concentrated in degrees zero and one, they belong to $\A_\X$.
Thus $\A_\X$ is abelian. 
Since every object of $\A_\X$ is finite-dimensional over $k$, every ascending or descending chain of subobjects terminates. Hence $\A_\X$ is of finite length.

(ii) If $\X$ contains a nonprojective module $X$, choose a nonzero direct summand $M$ of $X$ having no nonzero projective direct summand.
Lemma \ref{basic properties2}(2)--(3) gives $J^2M=0$ and a grading $M=M_0\oplus M_1$, so $0\ne M\in\A_\X$.

(iii) Suppose first that the underlying $\L$-module of $M$ has a simple direct summand $S$. Since $\X$ is closed under direct summands, one has $S\in\X$, contradicting (1). Hence $M$ has no simple direct summand.
Since $M$ is concentrated in degrees zero and one, one has $J^2M=0$ and $\rad M=JM_0\subseteq M_1\subseteq\soc M$.
Lemma \ref{basic properties2}(4) gives $M_1=\rad M=\soc M=JM_0$, and thus one has $M_0\ne0$ as $M\ne 0$.
If $\rad M=M_1=0$, then $M$ is a nonzero semisimple module, and hence has a simple direct summand, a contradiction.
\end{proof}

We now define the two operations that will be used throughout the rest of the paper. These correspond respectively to the operations of taking the minimal syzygy and cosyzygy as graded modules.
The notation for the invariant of a graded module defined below is motivated by the \textit{dimension vector}; see \cite{E}. In fact, when the module in Lemma \ref{basic properties2}(4) is endowed with a grading as in (2), the two notions coincide.

\begin{dfn}
For a graded module $M=M_0\oplus M_1$ concentrated in degrees zero and one, set
$$
\bft(M)=(\dim_ke_1M_0,\ldots,\dim_ke_rM_0)^{\T},\quad
\bfs(M)=(\dim_ke_1M_1,\ldots,\dim_ke_rM_1)^{\T}\in\Z_{\ge0}^r,
$$
and write $\dimvec M=(\bft(M)\mid\bfs(M))$. 
For $\bfu=(u_1,\ldots,u_r)^{\T}\in\Z_{\ge0}^r$, set $P(\bfu)=\bigoplus_{i=1}^rP_i^{\oplus u_i}$.
\end{dfn}

The numerical syzygy formula appearing in the next lemma is essentially the formula in \cite[Lemma 2.2]{E}, under the identification of dimension vectors described above. The assertion needed in this paper goes beyond this numerical calculation: after normalizing syzygies and cosyzygies by grading shifts, we realize them as mutually inverse exact autoequivalences of the abelian category $\A_\X$.

\begin{lem}\label{syzygy functors}
Let $\X$ be a proper thick subcategory of $\mod\L$ containing $\L$, fix a grading of $\L$ as in Lemma \ref{basic properties2}(1), and let $\A_\X$ be as in Lemma \ref{thick to abelian}(2).
Then there exist mutually inverse exact functors $\Omega, \Omega^{-1}:\A_\X\to\A_\X$ satisfying the following:
for $M=M_0\oplus M_1\in\A_\X$, there are exact sequences
$$
0\to (\Omega M)(-1)\to P(\bft)\to M\to0,
\quad
0\to M\to P(\bfs)(1)\to (\Omega^{-1}M)(1)\to0
$$
in $\gmod\L$. Moreover, for $\bft=\bft(M)$ and $\bfs=\bfs(M)$, one has
$$
\bft(\Omega M)=E\bft-\bfs,\quad \bfs(\Omega M)=\bft,\quad \bft(\Omega^{-1}M)=\bfs,
\quad \bfs(\Omega^{-1}M)=E\bfs-\bft.
$$
In particular, $M=0$ if and only if $\Omega M=0$ if and only if $\Omega^{-1}M=0$.
\end{lem}
\begin{proof}
For $M=M_0\oplus M_1\in\A_\X$, write $\bft=(t_1,\ldots,t_r)^{\T}=\bft(M)$ and $\bfs=(s_1,\ldots,s_r)^{\T}=\bfs(M)$.

We first construct $\Omega$. By Lemma \ref{thick to abelian}(2)(iii), one has $M_1=\rad M=JM_0$ when $M\ne0$, and the same equality is clear when $M=0$. Hence the natural morphism
$$
\varepsilon_M:\L\otimes_{\L_0}M_0\to M,
\quad
a\otimes x\mapsto ax,
$$
in $\gmod\L$ is surjective. 
Set $\Omega M=(\ker\varepsilon_M)(1)$.
Since $M_0\cong\bigoplus_{i=1}^r(\L_0 e_i)^{\oplus t_i}$ as left $\L_0$-modules, 
one has $\L\otimes_{\L_0}M_0\cong P(\bft)$ in $\gmod \L$. 
As the degree-zero component of $\varepsilon_M$ is an isomorphism, $\Omega M$ belongs to $\A_\X$. In particular, one has
\begin{align*}
&\dim_k e_i(\Omega M)_0=\dim_k e_i(\ker\varepsilon_M)_1=\dim_k e_iP(\bft)_1-\dim_k e_i M_1=\sum_{j=1}^r(\dim_k e_i \L_1e_j)t_j-s_i \\
&\dim_k e_i(\Omega M)_1=\dim_k e_i(\ker\varepsilon_M)_2=\dim_k e_iP(\bft)_2=\sum_{j=1}^r(\dim_k e_i \L_2e_j)t_j.
\end{align*}
Lemmas \ref{basic properties1}(2) and \ref{basic properties2}(1) give $\dim_k e_i \L_1e_j=E_{ij}$ and $\dim_k e_i \L_2e_j=\dim_k e_i S_j=\delta_{ij}$, which shows
$$
\bft(\Omega M)=E\bft-\bfs,
\quad
\bfs(\Omega M)=\bft.
$$
If $f:M\to N$ is a morphism in $\A_\X$, then there is a commutative diagram
$$
\xymatrix{
&0\ar[r]&(\Omega M)(-1)\ar[r]&\L\otimes_{\L_0}M_0\ar[r]^{\varepsilon_M}\ar[d]_{\id_{\L}\otimes f_0}
&M\ar[d]^f\ar[r] &0\\
&0\ar[r]&(\Omega N)(-1)\ar[r]&\L\otimes_{\L_0}N_0\ar[r]^{\varepsilon_N}
&N\ar[r] &0
}
$$
with exact rows, which induces a morphism $\Omega f:\Omega M\to\Omega N$. This defines a functor $\Omega$.
Since $\L_0$ is semisimple, the functor $\L\otimes_{\L_0}(-)_0$ is exact on $\A_\X$. So $\Omega$ is an exact functor on $\A_\X$ by the snake lemma.

We show that $\Omega$ is fully faithful and essentially surjective.
For fullness, let $g:\Omega M\to\Omega N$ in $\A_\X$. 
Lemma \ref{Matlis duality} yields that $\L\otimes_{\L_0} N_0\cong P(\bft(N))$ is a graded injective $\L$-module, and hence there exist $h, f$ and a commutative diagram
$$
\xymatrix{
&0\ar[r]&(\Omega M)(-1)\ar[r]\ar[d]_{g(-1)}&\L\otimes_{\L_0}M_0\ar[r]^{\varepsilon_M}\ar[d]_{h}
&M\ar[d]^f\ar[r] &0\\
&0\ar[r]&(\Omega N)(-1)\ar[r]&\L\otimes_{\L_0}N_0\ar[r]^{\varepsilon_N}
&N\ar[r] &0,
}
$$
in $\gmod\L$. As $(\varepsilon_N)_0$ is an isomorphism, the degree-zero component of $h$ and $\id_{\L}\otimes f_0$ coincide.
This implies that $h=\id_{\L}\otimes f_0$ since $\L\otimes_{\L_0}M_0$ is generated by the elements of degree zero. We get $g=\Omega f$.

For faithfulness, suppose $\Omega f=0$ for $f:M\to N$ in $\A_\X$.  
Then $\id_\L\otimes f_0$ factors through a morphism $u:M\to\L\otimes_{\L_0}N_0=:P$ in $\gmod\L$.
As $\soc M=M_1$ and $\soc P=P_2$, we have $u(M_1)\subseteq P_1\cap \soc P=0$.
Hence $Ju(M_0)=u(JM_0)=u(M_1)=0$, which means that $u(M_0)\subseteq P_0\cap \soc P=0$. Thus $u=0$, and hence $\id_\L\otimes f_0$ and $f$ are zero.

It remains to prove essential surjectivity.
One has $\soc(M(-1))=(M(-1))_2\cong P(\bfs)_2=\soc(P(\bfs))$.
Lemmas \ref{basic properties1}(2) and \ref{Matlis duality} give a
graded injective envelope 
$$
0\to M(-1)\to P(\bfs)\xrightarrow{q} L\to0.
$$
Then $q_0$ is an isomorphism, and $L\in\A_\X$. Hence, under the induced identification $\L\otimes_{\L_0}L_0\cong P(\bfs)$, the morphism $q$ is $\varepsilon_L$, since both maps agree in degree
zero. Therefore $\Omega L\cong M$.

Consequently $\Omega$ is an exact equivalence. Choose a quasi-inverse $\Omega^{-1}$ so that $\Omega^{-1}M$ is the module $L$ constructed above.
By full faithfulness these choices extend uniquely to a functor.
Shifting the defining sequence gives
$$
0\to M\to P(\bfs)(1)\to (\Omega^{-1}M)(1)\to0,
$$
and taking its degree-zero and degree-one components yields
$$
\bft(\Omega^{-1}M)=\bfs,
\quad
\bfs(\Omega^{-1}M)=E\bfs-\bft.
$$
The quasi-inverse of an exact equivalence is exact, and the final assertion follows immediately from the mutual inverse property.
\end{proof}

\section{The numerical recurrences and the case $\lambda\ne 2$}

We now turn from the discussion of graded modules to the numerical recurrence for the Betti numbers determined by the matrix $E$ and the syzygy and cosyzygy functors.
It is known from \cite{Ben, E, ESo, ESo2}, among others, that the behavior of algebras and modules changes substantially at the critical value two of the spectral radius.
In this section, by combining recurrence relations that are closely related to those in \cite{E} with the functors constructed in the preceding section, we determine the behavior of thick subcategories when the spectral radius is different from two.

\begin{nota}
Throughout this section, we use the same notation as in Notation \ref{notation}.
\end{nota}

We first prepare two lemmas on recurrence relations determined by $E$, which will be needed to state the main results of this section and the next.

\begin{lem}\label{zenkashki}
The following statements hold.
\begin{enumerate}[{\rm (1)}]
\item $\lambda$ is an eigenvalue of $E$ having multiplicity one, and there exists $\bfv\in\mathbb{R}_{>0}^r$ such that $E\bfv=\lambda\bfv$.
\item Assume that $\lambda>2$. Let $\bfv\in\mathbb{R}_{>0}^r$ with $E\bfv=\lambda\bfv$, let $\bfa,\bfb\in\mathbb{Z}_{\ge0}^r$ be nonzero, set
$\alpha=\bfv^{\T}\bfa$ and
$\beta=\bfv^{\T}\bfb$, and assume that
$\lambda\alpha\ge2\beta$ and $\lambda\beta\ge2\alpha$.
Define $\bfq_0=\bfa$, $\bfq_1=\bfb$, and
$$
\bfq_{n+1}=E\bfq_n-\bfq_{n-1}
$$
for $n\ge1$. Then there exists $n\ge2$ such that $\bfq_n^{\T}\bfa-\bfq_{n-1}^{\T}\bfb>0$.
\end{enumerate}
\end{lem}

\begin{proof}
(1) Lemma \ref{basic properties1}(3)--(4) shows that $E$ is a symmetric irreducible matrix with nonnegative integral entries. Hence the assertion follows from the Perron--Frobenius theorem.

(2) Set $\gamma=\lambda\alpha\beta-\alpha^2-\beta^2$. By assumption,
$$
\gamma=\frac{\alpha}{2}(\lambda\beta-2\alpha)+\frac{\beta}{2}(\lambda\alpha-2\beta)\ge0.
$$
Since $\alpha,\beta>0$, equality would imply $\lambda\beta=2\alpha$ and $\lambda\alpha=2\beta$, and hence $\lambda^2=4$, contrary to $\lambda>2$. Thus $\gamma>0$.
Since $E$ is real symmetric, we can choose an orthonormal basis $\bfz_1,\ldots,\bfz_r$ of $\mathbb{R}^r$ consisting of eigenvectors of $E$, with $\bfz_1=\bfv/\|\bfv\|$.
We write
$$
E\bfz_i=\mu_i\bfz_i, \quad \bfa=\sum_{i=1}^r a_i\bfz_i, \quad \bfb=\sum_{i=1}^r b_i\bfz_i.
$$
Then $\mu_1=\lambda$, $a_1=\alpha/\|\bfv\|$, and $b_1=\beta/\|\bfv\|$.
In particular,
\begin{equation}\label{rho coefficient}
\lambda a_1b_1-a_1^2-b_1^2=\gamma/\|\bfv\|^2>0.
\end{equation}
For a real number $x$, define
$$
f_{-1}(x)=0,\quad f_0(x)=1,\quad f_1(x)=x,\quad f_n(x)=xf_{n-1}(x)-f_{n-2}(x)\quad(n\ge2).
$$
The recurrence defining the $\bfq_n$ gives, by induction,
$$
\bfq_n=\sum_{i=1}^r\bigl(f_{n-1}(\mu_i)b_i-f_{n-2}(\mu_i)a_i\bigr)\bfz_i
$$
for every $n\ge1$.
Indeed, it is easy to see that the formula holds for $n=1,2$.
If the formula holds for $n$ and $n-1$, then
\begin{align*}
\bfq_{n+1}=E\bfq_n-\bfq_{n-1}&=\sum_{i=1}^r\mu_i\bigl(f_{n-1}(\mu_i)b_i-f_{n-2}(\mu_i)a_i\bigr)\bfz_i-\sum_{i=1}^r\bigl(f_{n-2}(\mu_i)b_i-f_{n-3}(\mu_i)a_i
\bigr)\bfz_i\\
&=\sum_{i=1}^r\Bigl[\bigl(\mu_i f_{n-1}(\mu_i)-f_{n-2}(\mu_i)\bigr)b_i-\bigl(\mu_i f_{n-2}(\mu_i)-f_{n-3}(\mu_i)\bigr)a_i\Bigr]\bfz_i\\
&=\sum_{i=1}^r\bigl(f_n(\mu_i)b_i-f_{n-1}(\mu_i)a_i\bigr)\bfz_i.
\end{align*}
For $n\ge2$, set
$$
d_n=\bfq_n^{\T}\bfa-\bfq_{n-1}^{\T}\bfb.
$$
Using the orthonormality of the $\bfz_i$ and the equality $f_{n-1}(x)+f_{n-3}(x)=x f_{n-2}(x)$, we obtain
\begin{align}\label{equality dn}
d_n&=\sum_{i=1}^r\bigl(f_{n-1}(\mu_i)b_i-f_{n-2}(\mu_i)a_i\bigr)a_i-\sum_{i=1}^r\bigl(f_{n-2}(\mu_i)b_i-f_{n-3}(\mu_i)a_i\bigr)b_i\\
&=\sum_{i=1}^r \bigl(f_{n-1}(\mu_i)+f_{n-3}(\mu_i)\bigr)a_ib_i-f_{n-2}(\mu_i)\bigl(a_i^2+b_i^2\bigr) \notag
\\  
&=\sum_{i=1}^r f_{n-2}(\mu_i)\bigl(\mu_i a_ib_i-a_i^2-b_i^2\bigr). \notag
\end{align}
We see that, for every $m\ge0$, there are nonnegative real numbers
$c_{m,1},\ldots,c_{m,m+1}$ such that
\begin{equation}\label{equality lambda rho}
(x+\lambda)^m=\sum_{j=1}^{m+1}c_{m,j}f_{j-1}(x)  \quad \text{holds for all $x$}.
\end{equation}
Indeed, this is clear for $m=0$ since $f_0(x)=1$, and if such $c_{m,1},\ldots,c_{m,m+1}$ exist for $m\ge 0$, then 
\begin{align*}
(x+\lambda)^{m+1}=(x+\lambda)\sum_{j=1}^{m+1}c_{m,j}f_{j-1}(x)&=\sum_{j=1}^{m+1}c_{m,j}\bigl(f_j(x)+f_{j-2}(x)\bigr)+\sum_{j=1}^{m+1}\lambda c_{m,j}f_{j-1}(x)\\
&=\sum_{j=1}^{m+2}\bigl(c_{m,j-1}+c_{m,j+1}+\lambda c_{m,j}\bigr)f_{j-1}(x)
\end{align*}
where $c_{m,0}=c_{m,m+2}=c_{m,m+3}=0$ since $x f_{j-1}(x)=f_j(x)+f_{j-2}(x)$.

Suppose that $d_n\le0$ for every $n\ge2$. Then, for every $m\ge0$,
$$
0\ge\sum_{j=1}^{m+1}c_{m,j}d_{j+1}=\sum_{j=1}^{m+1}c_{m,j}\sum_{i=1}^r f_{j-1}(\mu_i)\bigl(\mu_i a_ib_i-a_i^2-b_i^2\bigr)
=\sum_{i=1}^r(\mu_i+\lambda)^m\bigl(\mu_i a_ib_i-a_i^2-b_i^2\bigr)
$$
by \eqref{equality dn} and \eqref{equality lambda rho}.
Since $\mu_1=\lambda$, dividing by $(2\lambda)^m$ gives
$$
0\ge (\lambda a_1b_1-a_1^2-b_1^2)+\sum_{i=2}^r \left(\frac{\mu_i+\lambda}{2\lambda}\right)^m \bigl(\mu_i a_ib_i-a_i^2-b_i^2\bigr).
$$
Since $\lambda$ has multiplicity one and is the spectral radius of $E$, one has
$-\lambda\le\mu_i<\lambda$ for every $i\ge2$. Hence
$$
0\le\frac{\mu_i+\lambda}{2\lambda}<1.
$$
Letting $m\to\infty$ gives $0\ge \lambda a_1b_1-a_1^2-b_1^2$, contradicting \eqref{rho coefficient}. Thus $d_n>0$ for some $n\ge2$.
\end{proof}

\begin{lem}\label{two sided zenkashiki}
Let $(\bfq_m)_{m\in\Z}$ be a bi-infinite sequence of nonzero vectors in $\Z_{\ge0}^r$ satisfying
$$
\bfq_{m+1}=E\bfq_m-\bfq_{m-1}
$$
for every $m\in\Z$. Then $\lambda\ge2$.  If $\lambda=2$, then the set $\{\bfq_m\mid m\in\Z\}$ is finite.
\end{lem}

\begin{proof}
Let $\bfv\in\mathbb{R}_{>0}^r$ be as in Lemma \ref{zenkashki}(1), and set $x_m=\bfv^{\T}\bfq_m$ for each $m\in\Z$.
Since all the entries of $\bfv$ are positive, while $\bfq_m$ is nonzero and has only nonnegative integral entries, we have
$$
x_m=\bfv^{\T}\bfq_m=\sum_{i=1}^r v_iq_{m,i}\ge\min\{v_1,\ldots,v_r\}>0
$$
for every $m\in\Z$, where $\bfv=(v_1,\ldots,v_r)^{\T}$ and $\bfq_m=(q_{m,1},\ldots,q_{m,r})^{\T}$. One has
$$
x_{m+1}+x_{m-1}=\bfv^{\T}(\bfq_{m+1}+\bfq_{m-1})=\bfv^{\T}(E\bfq_m)=(E\bfv)^{\T}\bfq_m=\lambda x_m
$$
by Lemma \ref{basic properties1}(3).  Put $d_m=x_m-x_{m-1}$.  If $\lambda<2$, then,
$$
d_{m+1}-d_m=x_{m+1}+x_{m-1}-2x_m=(\lambda-2)x_m
\le (\lambda-2)\min\{v_1,\ldots,v_r\}<0.
$$
Thus $d_m$ tends to $-\infty$ at least linearly as $m\to\infty$, forcing $x_m<0$ for $m\gg0$, a contradiction.

If $\lambda=2$, then $d_{m+1}=d_m$ for all $m$, so $x_m=x_0+mc$ for some $c\in\mathbb R$.  Since $x_m>0$ for both positive and negative $m$, one has $c=0$.  Hence $x_m=x_0$ for all $m$.  Since every entry of $\bfv$ is positive, each entry of every $\bfq_m$ is bounded above by $x_0/\min\{v_1,\ldots,v_r\}$.  Thus only finitely many $\bfq_m\in\Z_{\ge0}^r$ can occur.
\end{proof}

We are now ready to prove the main results of this section. The following proposition and theorem show that, when the spectral radius is different from $2$, the only thick subcategory containing the algebra itself is the trivial one.

\begin{prop}\label{less than two}
If $\lambda<2$, the only thick subcategories of $\mod\L$ containing $\L$ are $\add\L$ and $\mod\L$.
\end{prop}

\begin{proof}
By Remark \ref{square zero}, we may assume that $J^2\ne 0$.
Suppose that there exists a thick subcategory $\X$ of $\mod\L$ such that $\add\L\subsetneq\X\subsetneq\mod\L$.
Fix a grading of $\L$ as in Lemma \ref{basic properties2}(1).
By Lemma \ref{thick to abelian}(2)(ii), there is a nonzero object $X\in\A_\X$, and Lemma \ref{syzygy functors} allows $\Omega$ and $\Omega^{-1}$ to be iterated in both directions inside $\A_\X$.
Set $\bfq_0=\bft(X)$ and $\bfq_1=\bfs(X)$. 
The formulas in Lemma \ref{syzygy functors} show, by successively applying $\Omega^{-1}$ and $\Omega$, that there is a bi-infinite sequence $(\bfq_m)_{m\in\Z}$ of vectors in $\Z_{\ge0}^r$ satisfying
$$
\bfq_{m+1}=E\bfq_m-\bfq_{m-1}
$$
for every $m\in\Z$. 
Note that these vectors are nonzero by Lemma \ref{thick to abelian}(2)(iii).
Hence Lemma \ref{two sided zenkashiki} states $\lambda\ge2$, contradicting the assumption $\lambda<2$.
\end{proof}

\begin{thm}\label{more than two}
If $\lambda>2$, then the only thick subcategories of $\mod\L$ containing $\L$ are $\add\L$ and $\mod\L$.
\end{thm}

\begin{proof}
Remark \ref{square zero} shows $J^2\ne 0$. Suppose, towards a contradiction, that there exists a thick subcategory $\X$ of $\mod\L$ such that $\add\L\subsetneq\X\subsetneq\mod\L$. Fix a grading of $\L$ as in Lemma \ref{basic properties2}(1). Let $\A_\X$ be as in Lemma \ref{thick to abelian}(2) and let $\bfv\in\mathbb{R}_{>0}^r$ be as in Lemma \ref{zenkashki}(1). For $L=L_0\oplus L_1\in\A_\X$, we set
$$
w(L)=\bfv^{\T}(\bft(L)+\bfs(L)).
$$
Choose a nonzero $X\in\A_\X$ with $w(X)$ minimal; such a choice exists since $\A_\X$ contains a nonzero module by Lemma \ref{thick to abelian}(2)(ii), all entries of $\bfv$ are positive, and all entries of $\bft(L)$ and $\bfs(L)$ are nonnegative integers for all $L\in\A_\X$. Set $\bft=\bft(X)$, $\bfs=\bfs(X)$, $\alpha=\bfv^{\T}\bft$, and $\beta=\bfv^{\T}\bfs$. By Lemma \ref{thick to abelian}(2)(iii), both $\bft$ and $\bfs$ are nonzero. As all entries of $\bfv$ are positive, we have $\alpha,\beta>0$. By Lemma \ref{syzygy functors},  $\Omega$ and $\Omega^{-1}$ are mutually inverse autoequivalences of $\A_\X$. Since $E$ is symmetric, Lemma \ref{syzygy functors} and the minimality of $X$ give
\begin{align*}
&\alpha+\beta=w(X)\le w(\Omega X)=\bfv^{\T}(E\bft-\bfs+\bft)=(E\bfv)^{\T}\bft-\bfv^{\T}\bfs+\bfv^{\T}\bft=\lambda\alpha-\beta+\alpha \\
&\alpha+\beta=w(X)\le w(\Omega^{-1}X)=\bfv^{\T}(\bfs+E\bfs-\bft)=\bfv^{\T}\bfs+(E\bfv)^{\T}\bfs-\bfv^{\T}\bft=\beta+\lambda\beta-\alpha,
\end{align*}
which means that $\lambda\alpha\ge2\beta$ and $\lambda\beta\ge2\alpha$.

Let $\bfq_n$ be the sequence defined recursively by $\bfq_0=\bft$, $\bfq_1=\bfs$, and $\bfq_{m+1}=E\bfq_m-\bfq_{m-1}$. Lemma \ref{zenkashki}(2) states that there exists $n\ge 2$ such that $\bfq_n^{\T}\bft-\bfq_{n-1}^{\T}\bfs>0$.
Choosing iterated cosyzygies $\Omega^{-m}X$ inductively by Lemma \ref{syzygy functors} with $\Omega^0X=X$, we obtain the equalities $\bft(\Omega^{-m}X)=\bfq_m$ and $\bfs(\Omega^{-m}X)=\bfq_{m+1}$ for all $m\ge 0$.
Splicing the cosyzygy sequences in Lemma \ref{syzygy functors} gives an exact sequence
$$
(P(\bfq_{n-1}))(-1)\to P(\bfq_n)\to \Omega^{-n}X\to0
$$
in $\gmod\L$. Applying $\gHom_\L(-,X)$ yields an exact sequence
$$
0\to \gHom_\L(\Omega^{-n}X,X)\to \gHom_\L(P(\bfq_n),X)\to \gHom_\L((P(\bfq_{n-1}))(-1),X).
$$
Set $\bfq_m=(q_{m,1},\ldots,q_{m,r})^{\T}$. The above exact sequence implies an inequality
\begin{align*}
\dim_k\gHom_\L(\Omega^{-n}X,X)&\ge\dim_k\gHom_\L(P(\bfq_n),X)- \dim_k\gHom_\L((P(\bfq_{n-1}))(-1),X) \\
&=\dim_k \bigoplus_{i=1}^r\gHom_\L(\L e_i,X)^{\oplus q_{n,i}}- \dim_k\bigoplus_{i=1}^r\gHom_\L(\L e_i,X(1))^{\oplus q_{n-1,i}}\\
&=\dim_k \bigoplus_{i=1}^r (e_iX_0)^{\oplus q_{n,i}}- \dim_k\bigoplus_{i=1}^r (e_iX_1)^{\oplus q_{n-1,i}}=\bfq_n^{\T}\bft-\bfq_{n-1}^{\T}\bfs>0.
\end{align*}
Hence there is a nonzero morphism $f:\Omega^{-n}X\to X$ in $\gmod\L$.

Put $K=\ker f$ and $C=\cok f$. By Proposition \ref{kernel cokernel}, the underlying modules of $K$ and $C$ belong to $\X$, and hence $K,C\in\A_\X$. 
Since $C$ is a quotient of $X$, every component of $\bft(C)+\bfs(C)$ is at most the corresponding component of $\bft(X)+\bfs(X)$. Moreover, since $f\ne0$, at least one of these inequalities is strict.
As all entries of $\bfv$ are positive, one has $w(C)<w(X)$. The minimality of $X$ gives $C=0$, and
$$
0\to K\to \Omega^{-n}X\to X\to0
$$
is exact. Applying Lemma \ref{syzygy functors} successively $n$ times, there is an exact sequence
$$
0\to \Omega^nK\to\Omega^n(\Omega^{-n}X)\to \Omega^nX\to0
$$
in $\A_\X$ with $\Omega^n(\Omega^{-n}X)\cong X$.
Since $\Omega^nX\ne0$ by Lemma \ref{syzygy functors}, one gets $w(\Omega^nK)<w(X)$ by the same argument as above. The minimality of $X$ gives $\Omega^nK=0$, and hence $K=0$ again by Lemma \ref{syzygy functors}. Therefore $\Omega^{-n}X\cong X$ in $\gmod\L$, and hence $\bfq_n=\bft$ and $\bfq_{n+1}=\bfs$.
Set $x_m=\bfv^{\T}\bfq_m$. Then
$$
x_{m+1}+x_{m-1}=\bfv^{\T}(\bfq_{m+1}+\bfq_{m-1})=\bfv^{\T} E\bfq_m=(E\bfv)^{\T} \bfq_m=(\lambda\bfv)^{\T} \bfq_m=\lambda x_m
$$
for $m\ge1$, and $x_n=x_0$, $x_{n+1}=x_1$.
Summing the above equality for $m=1,\ldots,n$, we obtain
$$
2\sum_{m=1}^n x_m=\sum_{m=1}^n(x_{m+1}+x_{m-1})=\sum_{m=1}^n\lambda x_m=\lambda\sum_{m=1}^n x_m.
$$
Since $\bfq_m\in\mathbb{Z}_{\ge 0}^r$ and $\bfv\in\mathbb{R}_{>0}^r$, one has $x_m\ge0$, while $x_1=\beta>0$. Hence $\sum_{m=1}^n x_m>0$, contradicting $\lambda>2$. The proof is now completed.
\end{proof}

Let $\underline{\mod}\L$ denote the \textit{stable module category} of $\L$; its objects are the finitely generated left $\L$-modules, and the morphism space from $M$ to $N$ is
$$
\underline{\Hom}_\L(M,N)=\Hom_\L(M,N)/\{\text{morphisms factoring through projective modules}\}.
$$
Since $\L$ is self-injective, $\mod\L$ is a Frobenius category and hence $\underline{\mod}\L$ is a triangulated category whose suspension functor is given by the cosyzygy functor $\Omega^{-1}$.
For a triangulated category, a \textit{thick subcategory} means a full triangulated subcategory closed under direct summands.
Using the standard correspondence between thick subcategories of $\mod\L$ containing $\L$ and thick subcategories of $\underline{\mod}\L$ (see \cite[Theorem 1 and Remark 2]{KS} and \cite[Theorem 2.1]{R}), the assertion follows from the results above.

\begin{cor}\label{stable corollary}
If $\lambda\ne 2$, then the stable module category $\underline{\mod}\L$ has no nonzero proper thick subcategory.
\end{cor}

Combining Proposition \ref{less than two} and Theorem \ref{more than two}, we see that the existence of a thick subcategory $\X$ with $\add\L\subsetneq\X\subsetneq\mod\L$ forces $\lambda=2$. This is reminiscent of \cite[Proposition 3.6]{E}, where Erdmann proved, under the hypotheses considered there, that the existence of a nonprojective Ext-finite module also forces $\lambda=2$.
The corollary below refines it.

\begin{cor}\label{Ext Tor rigidity}
If $\lambda\ne 2$, then the following statements hold.
\begin{enumerate}[{\rm (1)}]
\item
If $M,N\in\mod\L$ satisfy $\Ext_\L^n(M,N)=0$ for all $n\gg0$, then at least one of $M$ and $N$ is projective.
\item If $M\in\mod\L$ and $N\in\mod\L^{\rm op}$ satisfy $\Tor_n^\L(N,M)=0$ for all $n\gg0$, then at least one of $M$ and $N$ is projective on the respective side.
\end{enumerate}
\end{cor}

\begin{proof}
For (1), we consider the subcategory $\X=\{X\in\mod\L\mid \Ext_\L^n(X,N)=0\text{ for all }n\gg0\}$ that is thick and contains $\L$.
If $\X=\mod\L$, then $N$ is eventually Ext-orthogonal to every module. This means that $N$ has finite injective dimension.
As $\L$ is self-injective, $N$ is injective, and hence projective.
Otherwise, $\X$ is proper.
Proposition \ref{less than two} and Theorem \ref{more than two} show that $\X=\add\L$, and hence $M\in\X$ is projective.
This proves (1). The proof of (2) is similar.
\end{proof}

\section{The boundary case $\lambda=2$}

In this section, we consider the boundary case $\lambda=2$. This is the only case in which nontrivial proper thick subcategories containing $\L$ can occur, by the results of the preceding section. Our aim is to describe all of them.

\begin{nota}
Throughout this section, we use the same notation as in Notation \ref{notation} and assume in addition that $\lambda=2$. Fix a grading of $\L$ as in Lemma \ref{basic properties2}(1).
\end{nota}

As Benson \cite{Ben} showed under the hypotheses considered there, the spectral radius is closely related to the growth of minimal projective resolutions. The growth of Betti numbers also plays a crucial role in the study of thick subcategories considered in this paper. We begin by recalling some basic notions concerning Betti numbers.

\begin{dfn}
Since $\L$ is a finite-dimensional algebra, every finitely generated
left $\L$-module admits a projective cover, and hence a minimal projective
resolution; see \cite[Section 4]{I} for instance.
Let $M\in\mod\L$, and let $Q$ be a minimal projective resolution of $M$ with $Q_n\cong\bigoplus_{i=1}^rP_i^{\oplus\beta_{n,i}(M)}$.
We call
$$
\boldsymbol{\beta}_n(M)=(\beta_{n,1}(M),\ldots,\beta_{n,r}(M))^{\T}, \ \text{and} \ \beta_n(M)=\sum_{i=1}^r\beta_{n,i}(M)
$$
the $n$th \textit{Betti vector} of $M$, and the $n$th \textit{Betti number} of $M$, respectively.
We say that $M$ has \textit{bounded Betti numbers} if the sequence $(\beta_n(M))_{n\ge0}$ is bounded. We denote by $\BB(\L)$ the subcategory of $\mod\L$ consisting of modules having bounded Betti numbers.
\end{dfn}

The first step is to identify the part of the module category in which proper thick subcategories can live. At the critical value $\lambda=2$, boundedness of the Betti numbers gives precisely this categorical boundary.

\begin{prop}\label{bounded betti thick}
$\BB(\L)$ is the unique maximal proper thick subcategory of $\mod\L$ containing $\L$.
\end{prop}

\begin{proof}
Let $A=\L/J$ and $M\in\mod\L$. 
A minimal projective resolution $Q=(\cdots\to Q_2\xrightarrow{d_2} Q_1\xrightarrow{d_1} Q_0\to 0)$ of $M$ satisfies $\image d_n\subseteq\rad Q_{n-1}=JQ_{n-1}$ for all $n\ge 1$, and hence the differentials of $A\otimes_\L Q$ are zero; see \cite[Section 4]{I} for instance.
This implies
$$
\Tor_n^\L(A,M)\cong A\otimes_\L Q_n\cong\bigoplus_{i=1}^r(A\otimes_\L P_i)^{\oplus\beta_{n,i}(M)} \ \text{and}\  \dim_k\Tor_n^\L(A,M)=\sum_{i=1}^r(\dim_kS_i)\beta_{n,i}(M).
$$
So the boundedness of $\dim_k\Tor_n^\L(A,M)$ is equivalent to that of $\beta_n(M)$. Therefore the long exact sequence of $\Tor^\L(A,-)$ associated to a short exact sequence shows that boundedness of the Betti numbers satisfies the two-out-of-three property. It is also preserved under direct summands. Thus $\BB(\L)$ is thick.

Fix $1\le i\le r$, and using the notation of the preceding paragraph, set $M=M_0=S_i$ and $M_n=\image d_n$ for $n\ge1$.
Since $M_{n+1}\subseteq JQ_n$, restricting the exact sequence $Q_{n+1}\xrightarrow{d_{n+1}} Q_n\xrightarrow{d_n} Q_{n-1}$, we obtain an exact sequence
$$
Q_{n+1}\xrightarrow{d_{n+1}} JQ_n\xrightarrow{d_n\mid_{J Q_n}} J^2Q_{n-1}
$$ 
which induces the exact sequence $Q_{n+1}/J Q_{n+1}\to J Q_n/J^2Q_n\to J^2Q_{n-1}$ as $J^3=0$. One has
\begin{align*}
\beta_{n+1,j}(S_i)=\dim_k e_j (Q_{n+1}/J Q_{n+1})&\ge\dim_k e_j (J Q_n/J^2Q_n)-\dim_k e_j J^2Q_{n-1}\\ &=\sum_{h=1}^r E_{j h} \beta_{n,h}(S_i)-\beta_{n-1,j}(S_i).
\end{align*}
So each component of $\boldsymbol{\beta}_{n+1}(S_i)$ is more than or equal to that of $E\boldsymbol{\beta}_n(S_i)-\boldsymbol{\beta}_{n-1}(S_i)$.
Let $\bfv\in\mathbb{R}_{>0}^r$ be as in Lemma \ref{zenkashki}(1) and set $x_n=\bfv^{\T}\boldsymbol{\beta}_n(S_i)$.
As all entries of $\bfv$ are positive, $E$ is symmetric, and $\lambda=2$, 
$$
x_{n+1}=\bfv^{\T}\boldsymbol{\beta}_{n+1}(S_i)\ge \bfv^{\T}(E\boldsymbol{\beta}_n(S_i)-\boldsymbol{\beta}_{n-1}(S_i))=(E\bfv)^{\T}\boldsymbol{\beta}_n(S_i)-\bfv^{\T}\boldsymbol{\beta}_{n-1}(S_i)=2x_n-x_{n-1}
$$
holds. This shows that $x_{n+1}-x_n$ is nondecreasing in $n$.
Now $x_0=\bfv^{\T}\boldsymbol{\beta}_0(S_i)=v_i$ and $x_1=\bfv^{\T}\boldsymbol{\beta}_1(S_i)=\bfv^{\T}\boldsymbol{\beta}_0(J e_i)=\bfv^{\T}E\boldsymbol{\beta}_0(S_i)=2v_i$, where $v_i$ is the $i$-th component of $\bfv$.
Thus $x_{n+1}-x_n\ge x_1-x_0=v_i>0$ for every $n\ge0$, and hence
$$
\max\{v_1,\ldots,v_r\}\beta_n(S_i)\ge x_n\ge(n+1)v_i.
$$
Thus $S_i\notin\BB(\L)$, and consequently $\BB(\L)\subsetneq\mod\L$.

Finally, we show the maximality.
Let $\X$ be a proper thick subcategory of $\mod\L$ containing $\L$.
Let $M\in\X$, and write $M=P\oplus N$, where $P$ is projective and $N$ has no nonzero projective direct summand. We prove $N\in\BB(\L)$. If $N=0$, there is nothing to prove; otherwise, by Lemma \ref{basic properties2}(2)--(3), $N$ admits a grading such that it belongs to $\A_\X$ in Lemma \ref{thick to abelian}(2).
Lemmas \ref{syzygy functors} and \ref{thick to abelian}(2)(iii) give a bi-infinite sequence $(\bfq_m)_{m\in\Z}$ of nonzero vectors in $\Z_{\ge0}^r$ satisfying $\bfq_{m+1}=E\bfq_m-\bfq_{m-1}$ and $\bfq_{-n}=\boldsymbol{\beta}_n(N)$ for every $m\in\Z$ and $n\ge0$.
Lemma \ref{two sided zenkashiki} shows that the set $\{\boldsymbol{\beta}_n(N)\mid n\ge 0\}\subseteq\{\bfq_m\mid m\in\Z\}$ is finite, and thus $N$ belongs to $\BB(\L)$.  
\end{proof}

Proposition \ref{bounded betti thick} gives a categorical counterpart of Benson's bounded/linear growth dichotomy at spectral radius two \cite{Ben}: the modules of bounded growth form the unique maximal proper thick subcategory containing $\L$, whereas $\thick_{\mod\L}\{\L,M\}=\mod\L$ for every module $M$ of linear growth.
Thus the classification of proper thick subcategories is reduced to the bounded growth part $\BB(\L)$.

\begin{dfn}
Hereafter, for simplicity of notation, we write 
$$
\A:=\A_{\BB(\L)}.
$$
By Proposition \ref{bounded betti thick}, Lemma \ref{thick to abelian}(2), and Lemma \ref{syzygy functors}, the category $\A$ is an abelian category of finite length and $\Omega$ is an exact autoequivalence of $\A$ with inverse $\Omega^{-1}$. Of course, $\A$ depends on the fixed grading of $\L$, but it is unique up to equivalence as an abelian category; see Remark \ref{unique grading}.
\end{dfn}

We next give a numerical characterization of the indecomposable objects of $\A$. It also shows that, although the modules themselves need not be periodic under syzygy, their dimension vectors are periodic.

\begin{lem}\label{periodic}
Let $\bfv\in\mathbb R_{>0}^r$ satisfy $E\bfv=2\bfv$, and let $0\ne M=M_0\oplus M_1$ be an indecomposable graded $\L$-module concentrated in degrees zero and one.
Then $M\in\A$ if and only if $\bfv^{\T}\bft(M)=\bfv^{\T}\bfs(M)$.
If these equivalent conditions hold, then there exists $p>0$ such that $\dimvec(\Omega^n M)=\dimvec(\Omega^{n+p} M)$ for all $n\in\Z$.
\end{lem}

\begin{proof}
Suppose first that $M\in\A$.
Put $\bft_n=\bft(\Omega^nM)$, $\bfs_n=\bfs(\Omega^nM)$, and $x_n=\bfv^{\T}\bft_n$.
Lemma \ref{syzygy functors} gives
$$
\bft_{n+1}=E\bft_n-\bfs_n, \quad \bfs_{n+1}=\bft_n.
$$
Applying $\bfv^{\T}(-)$ to the former gives $x_{n+1}=2x_n-x_{n-1}$ for every $n\in\Z$.
Therefore $c:=x_n-x_{n-1}$ is constant on $n$ and we have $x_n=x_0+nc$.
Since $\Omega^nM$ is a nonzero object of $\A$ for every $n\in\Z$, Lemma \ref{thick to abelian}(2)(iii) gives $\bft_n\ne0$. Hence $x_n>0$ for every $n\in\Z$, which forces $c=0$. Thus
$$
\bfv^{\T}\bft(M)=x_0=x_{-1}=\bfv^{\T}\bfs(M).
$$
This proves the ``only if'' part. Before proving the converse, we establish the latter assertion of the lemma.
Since every entry of $\bfv$ is positive and $\bfv^{\T}\bft_n=x_n=x_0$ for all $n$, every entry of every $\bft_n$ is bounded.
The same is true for $\bfs_n=\bft_{n-1}$.
Hence only finitely many pairs $\dimvec(\Omega^n M)=(\bft_n\mid\bfs_n)$ occur. 
In particular, $\dimvec(\Omega^n M)=\dimvec(\Omega^{n+p} M)$ holds for some $p>0$ and $n\in\Z$. By Lemma \ref{syzygy functors}, both $\Omega$ and $\Omega^{-1}$ preserve equality with respect to $\dimvec$, and hence the desired assertion follows.

Conversely, suppose that $a:=\bfv^{\T}\bft(M)=\bfv^{\T}\bfs(M)$.
Since $\bfv\in\mathbb R_{>0}^r$ and $M\ne 0$, both $\bft(M)$ and $\bfs(M)$ are nonzero, and thus $M$ is not simple.
Since $M$ is concentrated in degrees zero and one, one has $J^2M=0$ and
$\rad M\subseteq M_1\subseteq\soc M$.
As $M$ is indecomposable, Lemma \ref{basic properties2}(4) gives $M_1=\rad M=\soc M$.
Thus the projective cover construction in the proof of Lemma \ref{syzygy functors} applies to $M$ and gives a nonzero normalized minimal syzygy, concentrated in degrees zero and one, with dimension vector $(E\bft-\bfs\mid\bft)$.
Moreover,
$$
\bfv^{\T}(E\bft-\bfs)=2a-a=a=\bfv^{\T}\bft.
$$
Since $\L$ is self-injective, a minimal syzygy of an indecomposable module is again indecomposable. Thus the same argument can be iterated, which means that $\bfv^{\T}\boldsymbol{\beta}_n(M)=a$ for every $n\ge 0$.
We obtain
$$
\min\{v_1,\ldots,v_r\}\beta_n(M)
\le \bfv^{\T}\boldsymbol{\beta}_n(M)
=a
$$
for every $n\ge0$.
Hence $M$ has bounded Betti numbers, and therefore $M\in\A$.
\end{proof}

We record the only numerical calculation needed for the classification.
In the proof of Theorem \ref{more than two}, the existence of a nonzero morphism was deduced from a numerical calculation of the dimension of Hom spaces.
Similarly, in the following lemma and proposition, numerical calculations of the dimensions of Hom and Ext spaces will be used to determine whether they vanish or not.

\begin{lem}\label{Ext calculation}
Let $M,N\in\A$, and write $\bft=\bft(M)$, $\bfs=\bfs(M)$, $\bfp=\bft(N)$, $\bfq=\bfs(N)$.
Then
$$
\dim_k\gHom_\L(M,N)-\dim_k\Ext_\L^1(M,N)_0=\bft^{\T}\bfp+\bfs^{\T}\bfq-\bft^{\T}E\bfq.
$$
\end{lem}

\begin{proof}
The first exact sequence in Lemma \ref{syzygy functors} gives $0\to (\Omega M)(-1)\to P(\bft)\to M\to0$. 
Applying $\Hom_\L(-,N)$ and taking the degree-zero part gives an exact sequence
$$
0\to\gHom_\L(M,N)\to\gHom_\L(P(\bft),N)
\to\gHom_\L((\Omega M)(-1),N)\to\Ext_\L^1(M,N)_0\to0.
$$
One has 
$$
\dim_k\gHom_\L(P(\bft),N)=\bft^{\T}\bfp.
$$
Moreover, $\bft(\Omega M)=E\bft-\bfs$ by Lemma \ref{syzygy functors}.  A degree-zero morphism $(\Omega M)(-1)\to N$ is uniquely determined by an $\L_0$-linear map $(\Omega M)_0\to N_1$. Since $\dim_k \Hom_{\L_0}(S_i, S_j)=\delta_{ij}$ for every $i,j$, we have
$$
\dim_k\gHom_\L((\Omega M)(-1),N)=\dim_k\gHom_{\L_0}((\Omega M)_0,N_1)=(E\bft-\bfs)^{\T}\bfq.
$$
The asserted equality follows from the symmetry of $E$.
\end{proof}

\begin{prop}\label{simple orthogonal}
Let $S,T$ be simple objects of $\A$.  If $\Ext_\L^1(S,T)_0\ne0$ or $\Ext_\L^1(T,S)_0\ne0$, then $T\cong \Omega^nS$ for some $n\in\Z$.
\end{prop}

\begin{proof}
Set $\bft_i=\bft(\Omega^iS)$, $\bfs_i=\bfs(\Omega^iS)$, $\bfp=\bft(T)$ and $\bfq=\bfs(T)$.
By Lemma \ref{periodic}, there is $p>0$ such that $\dimvec(\Omega^pS)=\dimvec S$.
Lemma \ref{syzygy functors} gives $\bfs_{i+1}=\bft_i$ and $\bft_{i+1}=E\bft_i-\bfs_i$.
Summing these equalities for $i=0,\ldots,p-1$ and using $\dimvec(\Omega^pS)=\dimvec S$, we obtain
\begin{equation}\label{about u}
\sum_{i=0}^{p-1}\bfs_i=\sum_{i=0}^{p-1}\bft_i=:\bfu, \quad E\bfu=2\bfu. 
\end{equation}
Suppose that $\Omega^i S\not\cong T$ for every $i$.
As $\Omega^i$ are exact autoequivalences of $\A$, $\Omega^i S$ are simple objects of $\A$ and hence $\gHom_\L(\Omega^iS,T)=0=\gHom_\L(T,\Omega^iS)$.
Lemma \ref{Ext calculation} gives
$$
-\dim_k\Ext_\L^1(\Omega^iS,T)_0=\bft_i^{\T}\bfp+\bfs_i^{\T}\bfq-\bft_i^{\T}E\bfq, \quad
-\dim_k\Ext_\L^1(T, \Omega^iS)_0=\bfp^{\T}\bft_i+\bfq^{\T}\bfs_i-\bfp^{\T}E\bfs_i.
$$
Summing both equalities for $i=0,\ldots,p-1$, together with the symmetry of $E$ and \eqref{about u}, we obtain
\begin{align*}
&-\sum_{i=0}^{p-1}\left(\dim_k\Ext_\L^1(\Omega^iS,T)_0+\dim_k\Ext_\L^1(T,\Omega^iS)_0
\right)\\
&=\sum_{i=0}^{p-1} \left(\bft_i^{\T}\bfp+\bfs_i^{\T}\bfq-\bft_i^{\T}E\bfq\right) +\sum_{i=0}^{p-1}\left(\bfp^{\T}\bft_i+\bfq^{\T}\bfs_i-\bfp^{\T}E\bfs_i\right)\\
&=(\bfu^{\T}\bfp+\bfu^{\T}\bfq-\bfu^{\T}E\bfq)+(\bfp^{\T}\bfu+\bfq^{\T}\bfu-\bfp^{\T}E\bfu)=\bfu^{\T}\bfp+\bfu^{\T}\bfq-2\bfu^{\T}\bfq+\bfp^{\T}\bfu+\bfq^{\T}\bfu-2\bfp^{\T}\bfu=0.
\end{align*}
This is impossible, since all the summands are nonnegative and, by assumption, at least one of the two summands for $i=0$ is positive. Hence $T\cong\Omega^nS$ for some $n\in\Z$.
\end{proof}

This proposition shows that simple objects should be distinguished according to whether or not they are related by syzygies.
Moreover, if two simple objects are orthogonal, then the modules obtained from them by successive extensions are also orthogonal. Thus, the orthogonality extends to the level of the Serre subcategories they generate.

\begin{dfn}
Let $\C$ be an abelian category.
We say that a subcategory $\X$ of $\C$ is a {\em Serre subcategory} if it is closed under subobjects, quotient objects, and extensions.
The \textit{Serre closure} of a subcategory $\X$ of $\C$ is defined to be the smallest Serre subcategory containing $\X$, and denoted $\Serre_{\C}(\X)$.
Let $\operatorname{Sim}(\A)$ denote the set of isomorphism classes of simple objects of $\A$. Since $\Omega$ is an exact autoequivalence of $\A$, it induces a permutation of $\operatorname{Sim}(\A)$. We call the orbits of this permutation the \textit{$\Omega$-orbits}. Thus two simple objects $S,T$ belong to the same $\Omega$-orbit if and only if $T\cong\Omega^nS$ for some $n\in\mathbb Z$.
\end{dfn}

\begin{lem}\label{orbit decomposition}
The following statements hold.
\begin{enumerate}[{\rm (1)}]
\item If $\O$ and $\O'$ are distinct $\Omega$-orbits and $M\in\Serre_{\A}(\O)$, $N\in\Serre_{\A}(\O')$, then $$\gHom_\L(M,N)=0=\Ext_\L^1(M,N)_0.$$
\item Every object $M\in\A$ admits a decomposition
$$
M\cong M_{\O_1}\oplus\cdots\oplus M_{\O_n}
$$
for some $n\ge0$ and some $\Omega$-orbits $\O_1,\ldots,\O_n$, where $M_{\O_i}\in\Serre_{\A}(\O_i)$ for each $1\le i\le n$.
\end{enumerate}
\end{lem}

\begin{proof}
(1) If $M$ and $N$ are simple objects belonging to distinct $\Omega$-orbits $\O$ and $\O'$, respectively, then $\gHom_\L(M,N)=0$, while Proposition \ref{simple orthogonal} gives $\Ext_\L^1(M,N)_0=0$. For general $M$ and $N$, the assertion follows by induction on the sum of their composition lengths.

(2) We use induction on the length of $M$. The assertion is clear when $M=0$ or $M$ is simple. Choose an exact sequence $0\to L\to M\to S\to0$ in $\A$ with $S\in\operatorname{Sim}(\A)$, and let $\O$ be the $\Omega$-orbit of $S$. By induction and by combining summands belonging to the same orbit, we may write $L=L_{\O}\oplus L'$, where $L_{\O}\in\Serre_{\A}(\O)$ and every nonzero direct summand of $L'$ belongs to an $\Omega$-orbit distinct from $\O$. From the exact sequence $0\to L_{\O}\oplus L'\to M\to S\to0$, one obtains an object $K\in\A$ fitting into exact sequences
$$
0\to L_{\O}\to K\to S\to0, \quad 0\to L'\to M\to K\to0;
$$
the construction is the same as in the proof of \cite[Lemma 3.1]{T06}. Then $K$ belongs to $\Serre_{\A}(\O)$, and hence $\Ext_\L^1(K,L')_0=0$ by (1). Lemma \ref{graded Ext1} shows that $0\to L'\to M\to K\to0$ splits, which gives the desired decomposition.
\end{proof}

The following lemma gives one part of the main result of this section, namely that a proper thick subcategory of $\mod\L$ containing $\L$ gives rise to a Serre subcategory of $\A$.

\begin{lem}\label{Serre}
Let $\X$ be a proper thick subcategory of $\mod\L$ containing $\L$.  Then $\A_\X$ of Lemma \ref{thick to abelian}(2) is a Serre subcategory of $\A$.
\end{lem}

\begin{proof}
Proposition \ref{bounded betti thick} gives $\X\subseteq\BB(\L)$, and hence $\A_\X\subseteq\A$.  Lemma \ref{thick to abelian}(2)(i) shows that $\A_\X$ is an abelian subcategory of $\A$ of finite length, and  it is stable under $\Omega$ and $\Omega^{-1}$ by Lemma \ref{syzygy functors}.

We prove that every simple object of $\A_\X$ is simple in $\A$.
Let $X$ be a simple object of $\A_\X$. Since it is indecomposable in $\A$, Lemma \ref{orbit decomposition}(2) shows that all simple composition factors of $X$ in $\A$ belong to one $\Omega$-orbit $\O$.
Suppose that $X$ is not simple in $\A$.  Since $\A$ has finite length, there are a simple subobject $S\subsetneq X$ and a simple quotient $X\twoheadrightarrow T$.  Both $S$ and $T$ belong to $\O$, so we can choose $n\in\Z$ such that $\Omega^nT\cong S$.  Applying the exact autoequivalence $\Omega^n$ to $X\twoheadrightarrow T$ and composing with $\Omega^nT\cong S\subseteq X$ gives a nonzero morphism $\Omega^nX\to X$.
Both $\Omega^nX$ and $X$ are simple objects of $\A_\X$, so this morphism is an isomorphism in $\A_\X$, and hence in $\A$.  On the other hand, its image is contained in the proper subobject $S$ of $X$, a contradiction. 

Let $M\in\A_\X$.  A composition series of $M$ in $\A_\X$ is also a composition series in $\A$, and all its factors belong to $\A_\X$.  If $L\subseteq M$ is any subobject in $\A$, intersecting this composition series with $L$ gives a filtration of $L$ whose nonzero factors are among the simple factors of $M$.  Thus $L\in\A_\X$ by extension closure.  The same argument applied to the quotient filtration gives $M/L\in\A_\X$.
It is easy to see that $\A_\X$ is closed under extensions in $\A$.
Hence $\A_\X$ is Serre.
\end{proof}

We pass once to the stable category in order to forget the grading. 

\begin{dfn}
Recall from the preceding section that $\underline{\mod}\L$ is triangulated and its suspension is the cosyzygy functor.  Since $\BB(\L)$ is thick and contains the projective modules, its image in $\underline{\mod}\L$ is a thick subcategory; we denote it by $\underline{\BB}(\L)$.
For an $\Omega$-orbit $\O$, let $\underline{\Serre}(\O)$ be the full subcategory of $\underline{\BB}(\L)$ consisting of the objects isomorphic in $\underline{\mod}\L$ to the underlying modules of objects of $\Serre_{\A}(\O)$.
\end{dfn}

\begin{prop}\label{stable orbit decomposition}
The following statements hold.
\begin{enumerate}[{\rm (1)}]
\item If $\O$ and $\O'$ are distinct $\Omega$-orbits, then $\underline{\Hom}_\L(M,N)=0$ for every $M\in\underline{\Serre}(\O)$ and $N\in\underline{\Serre}(\O')$.
\item Every object $M\in\underline{\BB}(\L)$ admits a decomposition 
$$
M\cong M_{\O_1}\oplus\cdots\oplus M_{\O_n}
$$ 
for some $n\ge0$ and some $\Omega$-orbits $\O_1,\ldots,\O_n$ with $M_{\O_i}\in\underline{\Serre}(\O_i)$ for each $1\le i\le n$.
\item Each $\underline{\Serre}(\O)$ is a nonzero thick subcategory of
$\underline{\mod}\L$.
\end{enumerate}
\end{prop}

\begin{proof}
(1)   It is enough to take nonzero modules $M\in\Serre_{\A}(\O)$ and $N\in\Serre_{\A}(\O')$ with $\O\ne\O'$ and prove $\underline{\Hom}_\L(M,N)=0$. 
Since every ungraded morphism $M\to N$ is the sum of its homogeneous components, it is enough to consider morphisms $f:M\to N(i)$ in $\gmod\L$ for all $i\in\Z$. Since both modules are concentrated in degrees zero and one, only $i=-1,0,1$ can occur.
Lemma \ref{orbit decomposition}(1) yields $f=0$ if $i=0$.
When $i=-1$, we have $f(M_0)\subseteq N(-1)_0=N_{-1}=0$ and hence $f(M_1)=f(JM_0)=Jf(M_0)=0$ by Lemma \ref{thick to abelian}(2)(iii), which means that $f=0$.
It remains to consider the case where $i=1$. Let 
$$
0\to (\Omega N)(-1)\to P\to N\to0,
\quad
P=P(\bft(N)).
$$
be an exact sequence in Lemma \ref{syzygy functors}.
Applying $\Hom_\L(M,-)$ and taking the degree one component gives an exact sequence
$$
\Hom_\L(M,P)_1\to\Hom_\L(M,N)_1 \to\Ext_\L^1(M,(\Omega N)(-1))_1\cong\Ext_\L^1(M,\Omega N)_0=0,
$$
where the last equality follows from Lemma \ref{orbit decomposition}(1).
Hence $f$ factors through a projective module $P$, and it is zero in $\underline{\Hom}_\L(M,N)$.

(2) Let $M\in\BB(\L)$ be a module with no nonzero projective summands. Lemma \ref{basic properties2}(2)--(3) gives a grading $M=M_0\oplus M_1$ with $M_1=JM$, so that $M\in\A$. By Lemma \ref{orbit decomposition}(2), there is a decomposition
$$
M\cong M_{\O_1}\oplus\cdots\oplus M_{\O_n}
$$
in $\A$, where $\O_1,\ldots,\O_n$ are $\Omega$-orbits and $M_{\O_i}\in\Serre_{\A}(\O_i)$ for each $i$. Forgetting the grading and passing to the stable category gives the desired decomposition.

(3) First of all, $\underline{\Serre}(\O)$ is nonzero.  Indeed, for every simple object $S\in\O$, its underlying module satisfies $J^2S=0$, whereas every nonzero projective $\L$-module is not annihilated by $J^2$; see Lemma \ref{basic properties1}(2).
The exact sequences in Lemma \ref{syzygy functors} show that syzygy and cosyzygy carry $\underline{\Serre}(\O)$ to itself, since $\Omega$ and $\Omega^{-1}$ preserve $\Serre_{\A}(\O)$. Hence $\underline{\Serre}(\O)$ is closed under the suspension and its inverse. Let
$$
X\longrightarrow Y\longrightarrow Z\longrightarrow\Omega^{-1}X
$$
be a triangle in $\underline{\mod}\L$ with $X,Y\in\underline{\Serre}(\O)$. 
Since $\underline{\BB}(\L)$ is thick and contains $X$ and $Y$, one has $Z\in\underline{\BB}(\L)$. By (2), there is a decomposition $Z\cong Z_{\O_1}\oplus\cdots\oplus Z_{\O_n}$, where $Z_{\O_i}\in\underline{\Serre}(\O_i)$. 
For $1\le i\le n$ with $\O_i\ne\O$, applying $\underline{\Hom}_\L(Z_{\O_i},-)$ to the triangle and using (1) gives
$$
0=\underline{\Hom}_\L(Z_{\O_i},Y)\to\underline{\Hom}_\L(Z_{\O_i},Z)\to \underline{\Hom}_\L(Z_{\O_i},\Omega^{-1}X)=0.
$$
Thus $\underline{\Hom}_\L(Z_{\O_i},Z)=0$, which implies that $Z_{\O_i}=0$ since $Z_{\O_i}$ is a direct summand of $Z$.
We get $Z=Z_\O\in \underline{\Serre}(\O)$.
Similarly, $\underline{\Serre}(\O)$ is closed under direct summands. Indeed, since $\underline{\BB}(\L)$ is thick, every direct summand of an object in $\underline{\Serre}(\O)$ belongs to $\underline{\BB}(\L)$, and (1) shows that, in the decomposition given by (2), no component other than the $\underline{\Serre}(\O)$-component can occur.
\end{proof}

We can now state the classification which is the main result of this section.
Let $\pi:\mod\L \to \underline{\mod}(\L)$ be the canonical functor.
When taking the thick closure of objects of $\A$, we always forget their gradings.

\begin{thm}\label{classification equal two}
There is a commutative diagram of order-preserving mutually inverse bijections
$$
\xymatrix@C=1.7pc{
{\left\{
\begin{matrix}
\text{$\Omega$-stable Serre}\\
\text{subcategories of $\A$}
\end{matrix}
\right\}}
\ar@<.7mm>[rr]^-{\thick((-)\cup\{\L\})}
\ar@<.7mm>[d]^-{(-)\cap\operatorname{Sim}(\A)}
&&
{\left\{
\begin{matrix}
\text{Proper thick subcategories}\\
\text{of $\mod\L$ containing $\L$}
\end{matrix}
\right\}}
\ar@<.7mm>[ll]^-{\A_{(-)}}
\ar@<.7mm>[r]^-{\pi}
\ar@{=}[d]
&
{\left\{
\begin{matrix}
\text{Proper thick subcategories}\\
\text{of $\underline{\mod}(\L)$}
\end{matrix}
\right\}}
\ar@<.7mm>[l]^-{\pi^{-1}}
\ar@{=}[d]\\
{\left\{
\begin{matrix}
\text{$\Omega$-stable subsets}\\
\text{of $\operatorname{Sim}(\A)$}
\end{matrix}
\right\}}
\ar@<.7mm>[rr]^-{\thick((-)\cup\{\L\})}
\ar@<.7mm>[u]^-{\Serre_{\A}(-)}
&&
{\left\{
\begin{matrix}
\text{Thick subcategories of $\mod\L$}\\
\text{between $\add\L$ and $\BB(\L)$}
\end{matrix}
\right\}}
\ar@<.7mm>[ll]^-{\A_{(-)}\cap\operatorname{Sim}(\A)}
\ar@<.7mm>[r]^-{\pi}
&
{\left\{
\begin{matrix}
\text{Thick subcategories}\\
\text{of $\underline{\BB}(\L)$}
\end{matrix}
\right\}}
\ar@<.7mm>[l]^-{\pi^{-1}}
}
$$
\end{thm}

\begin{proof}
The correspondence given by $\pi$ and $\pi^{-1}$ follows from \cite[Theorem 1 and Remark 2]{KS} and \cite[Theorem 2.1]{R}. The middle and right vertical equalities follow from Proposition \ref{bounded betti thick}.
Hence the right-hand square is a commutative diagram of bijections.
Since $\Omega$ is an exact autoequivalence of $\A$, the left vertical correspondence is obvious.
By Lemmas \ref{Serre} and \ref{syzygy functors}, $\A_{(-)}$ is well-defined, and by Proposition \ref{bounded betti thick},  $\thick((-)\cup\{\L\})$ is well-defined. Hence all the maps appearing in the diagram are well-defined.
The two composites from the upper middle term to the lower left term coincide, and so do the two composites in the opposite direction. Hence it remains only to show that the two maps between the upper left and upper middle terms are mutually inverse.

Let $\X$ be a proper thick subcategory of $\mod\L$ containing $\L$.
We claim that $\X=\thick(\A_\X\cup\{\L\})$.
The inclusion $\X\supseteq\thick(\A_\X\cup\{\L\})$ is clear.
For the other, let $M\in\X$ and write $M=P\oplus N$, where $P$ is projective and $N$ has no nonzero projective direct summand.
If $N\ne0$, Lemma \ref{basic properties2}(2)--(3) gives a grading of $N$ for which $N\in\A_\X$.
Thus $M\in\thick(\A_\X\cup\{\L\})$, and the claim follows.

Conversely, let $\mathcal{S}$ be an $\Omega$-stable Serre subcategory of $\A$ and set $\X=\thick(\mathcal{S}\cup\{\L\})$.
Clearly $\mathcal{S}\subseteq\A_\X$.
Set $W=\mathcal{S}\cap\operatorname{Sim}(\A)$.
Then $W$ is an $\Omega$-stable subset of $\operatorname{Sim}(\A)$ and $\mathcal{S}=\Serre_{\A} (W)$.
Set
$$
\mathcal{C}_W=\add_{\underline{\BB}(\L)}\{\underline{\Serre}(\O)\mid \O\subseteq W\}.
$$
By Proposition \ref{stable orbit decomposition}(1)--(3), $\mathcal{C}_W$ is thick: morphisms between distinct orbit components vanish, so cones and direct summands decompose orbitwise.
Since $\pi(\mathcal{S})\subseteq\mathcal{C}_W$ and $\pi(\L)=0$, we have
$\pi(\X)\subseteq\mathcal{C}_W$.
Let $M\in\A_\X$, and write
$$
M\cong M_{\O_1}\oplus\cdots\oplus M_{\O_n},
\qquad
M_{\O_i}\in\Serre_{\A}(\O_i)
$$
as in Lemma \ref{orbit decomposition}(2).
Since $\pi(M)\in\mathcal{C}_W$ and $\mathcal{C}_W$ is closed under direct
summands, $\pi(M_{\O_i})\in\mathcal{C}_W$ for every $i$.
If $M_{\O_i}\ne0$ and $\O_i\not\subseteq W$, Proposition
\ref{stable orbit decomposition}(1) gives $\underline{\Hom}_\L(\pi(M_{\O_i}),\mathcal{C}_W)=0$, and hence $\pi(M_{\O_i})=0$.
This is impossible, since a nonzero object of $\A$ is annihilated by
$J^2$, whereas no nonzero projective module is, by Lemma
\ref{basic properties1}(2).
Thus $\O_i\subseteq W$ whenever $M_{\O_i}\ne0$, so
$M\in\Serre_{\A}(W)=\mathcal{S}$.
Therefore $\A_\X=\mathcal{S}$.
\end{proof}

\begin{rem}\label{hereditary}
For this remark, assume in addition that $k$ is algebraically closed.
The category of graded $\L$-modules concentrated in degrees zero and one is naturally equivalent to $\mod H$, where
$$
H=
\begin{pmatrix}
\L_0&0\\
\L_1&\L_0
\end{pmatrix}.
$$
The algebra $H$ is hereditary, and after passing to its basic algebra, its quiver is the separated quiver of $\L/J^2$; see \cite[Chapter X, Section 2]{ARS}. 
Since $\lambda=2$, every connected component of the separated quiver is of Euclidean type by \cite[Theorem 4.1]{Ben}.
Let $\reg H$ denote the category of regular $H$-modules, and choose $\bfv\in\mathbb R_{>0}^r$ with $E\bfv=2\bfv$.
Since the adjacency matrix of the underlying separated graph is
$$
\begin{pmatrix}
0&E\\
E&0
\end{pmatrix},
$$
$(\bfv\mid\bfv)$ restricts on each component to a positive multiple of its radical vector; see \cite[Theorem 4.2.1]{Kr}, and the corresponding defect of $(\bft\mid\bfs)$ is $\bfv^{\T}\bft-\bfv^{\T}\bfs$ up to a nonzero scalar. Hence \cite[Proposition 5.2.1]{Kr} and Lemma \ref{periodic} show that the above equivalence restricts to $\A\cong\reg H$.
Consequently, $\Omega$ induces an exact autoequivalence $\sigma$ of $\reg H$, and Theorem \ref{classification equal two} identifies the proper thick subcategories of $\underline{\mod}(\L)$ with the $\sigma$-stable Serre subcategories of $\reg H$.
\end{rem}

\section{The general case}

In this section, we state the results in the general case obtained by combining the results for indecomposable algebras established in the preceding sections.
Let $k$ be a field, and let $\L$ be a finite-dimensional split weakly symmetric $k$-algebra with radical cube zero, and write
$$
\L=\L^{(1)}\times\cdots\times\L^{(t)}
$$
for its block decomposition. 
For each $1\le i\le t$, choose representatives $S_{i,1},\ldots,S_{i,r_i}$ of the isomorphism classes of simple $\L^{(i)}$-modules, and set
\begin{align*}
&E_i=\left(\dim_k\Ext_{\L^{(i)}}^1(S_{i,p},S_{i,q})\right)_{1\le p,q\le r_i},\quad \lambda_i=\rho(E_i), \\
&I_2=\{i\mid \lambda_i=2\}, \quad I_{\ne2}=\{i\mid \L^{(i)}\text{ is nonsemisimple and }\lambda_i\ne2\}.
\end{align*}
For each $i\in I_2$, fix a grading of $\L^{(i)}$ as in Lemma \ref{basic properties2}(1), and write $\A_i:=\A_{\BB(\L^{(i)})}$.
Let $\Omega_i$ be the exact autoequivalence of $\A_i$ in Lemma \ref{syzygy functors}, and let
$$
\mathfrak{S}_i=\{\text{$\Omega_i$-stable Serre subcategories of $\A_i$}\}\sqcup\{\gmod\L^{(i)}\}.
$$

\begin{thm}\label{general classification}
There are order-preserving bijections
$$
\left\{
\begin{matrix}
\text{thick subcategories of}\\
\text{$\mod\L$ containing $\L$}
\end{matrix}
\right\}
\longleftrightarrow
\mathcal{P}(I_{\ne2})
\times
\prod_{i\in I_2}\mathfrak{S}_i
\longleftrightarrow
\left\{
\begin{matrix}
\text{thick subcategories}\\
\text{of $\underline{\mod}\L$}
\end{matrix}
\right\},
$$
where $\mathcal{P}$ denotes the power set.
\end{thm}

\begin{proof}
Each block $\L^{(i)}$ is again a finite-dimensional split indecomposable weakly symmetric $k$-algebra with radical cube zero. Thus the results of Sections 4 and 5 apply to each $\L^{(i)}$.
Moreover,
$$
\mod\L\cong\prod_{i=1}^t\mod\L^{(i)}.
$$
Thus, the classification of thick subcategories over $\L$ reduces to combining the classifications for its blocks, and the assertion follows from Remark \ref{square zero}, Proposition \ref{less than two}, and Theorems \ref{more than two} and \ref{classification equal two}.
\end{proof}

\section{Commutative rings}

Some of the results in this paper remain meaningful even when restricted to commutative rings; in fact, some of their most important applications arise in this setting.
In this section, we discuss these applications and place our results in the context of previous work on commutative rings.
We first recall the notion of dominance for commutative Noetherian local rings introduced by Takahashi \cite{T23}.

\begin{dfn}
Suppose that $R$ is a commutative Noetherian local ring with residue field $k$.
Following \cite[Corollary 10.8 and Remark 10.9]{T23}, we say that $R$ is \textit{dominant} provided that $k$ belongs to $\thick_{\mod R}\{R, M\}$ for every finitely generated $R$-module $M$ with infinite projective dimension.
\end{dfn}

Takahashi proved that, over certain commutative rings, thick and resolving subcategories of module, derived, and singularity categories can be classified in terms of specialization-closed subsets of the spectrum, and introduced dominance as a framework for rings over which such classifications are possible.
By \cite[Proposition 5.10]{T23}, hypersurfaces, Cohen--Macaulay local rings of minimal multiplicity with infinite residue field, local rings whose maximal ideals are quasi-decomposable, and Burch local rings are dominant.
Moreover, under the assumption that the ring is Cohen--Macaulay and complete with infinite residue field, \cite[Theorem 1.2]{KT} establishes dominance for non-complete intersection rings of codimension two; non-Gorenstein rings with $e(R)\le 5$; non-Gorenstein G-regular rings with $e(R)\le 6$; non-Gorenstein rings of codimension two with $e(R)\le 11$; and, when the residue field is algebraically closed, rings of finite Cohen--Macaulay representation type and dimension at most two.
Whether every Golod local ring is dominant remains open, as posed in \cite[Question 9.4]{T23} and \cite[Question 1.1(3)]{KT}.

In this way, the dominance of various classes of rings has been studied.
However, the properties of the rings listed above are either incompatible with the Gorenstein property or, in the presence of the Gorenstein property, force the ring to be a hypersurface. 
Accordingly, the following has been posed as one of the central questions in this line of research.
For a commutative Noetherian local ring $(R,\m,k)$, we denote by $e(R)$ the \textit{(Hilbert--Samuel) multiplicity} of $R$, and define the \textit{embedding dimension} of $R$ by $\edim R=\dim_k \m/\m^2$ and the \textit{codimension} of $R$ by $\codim R=\edim R-\dim R$.

\begin{ques}\label{ques dom Gor}
\begin{enumerate}[\rm (1)]
\item \cite[Question 9.6]{T23} Is a Gorenstein dominant local ring a hypersurface?
\item \cite[Question 7.1]{KT} Let $R$ be a Gorenstein local ring with maximal ideal $\m$. Suppose that $R$ is not a complete intersection. If either $\m^3=0$ or $e(R)\le \codim R+2$, is then $R$ necessarily dominant?
\end{enumerate}
\end{ques}

In particular, a negative answer to the former question was desirable also from the viewpoint of the ubiquity of dominance. In other words, it is of considerable significance to find a Gorenstein dominant local ring that is not a hypersurface. By \cite[Proposition 6.2(3)]{T23}, such a ring is automatically not a complete intersection. The case $\lambda>2$ of our results, when specialized to commutative rings, provides precisely such examples. Indeed, for an Artinian ring $R$, being dominant is equivalent to the condition that the only thick subcategories of $\mod R$ containing $R$ are $\add R$ and $\mod R$.

\begin{thm}\label{main commutative}
\begin{enumerate}[{\rm (1)}]
\item Every commutative Noetherian Gorenstein local ring containing a field is dominant if it is not a complete intersection and its maximal ideal has cube zero.
\item Every equicharacteristic commutative Noetherian Gorenstein local ring $R$ with an infinite residue field that is not a complete intersection and satisfies $e(R)\le \codim R+2$ is dominant.
\end{enumerate}
\end{thm}

\begin{proof}
(1) Let $(R,\m,k)$ be such a commutative local ring. Since $R$ is commutative, local, Gorenstein, and $\m^3=0$, it admits a coefficient field and hence it is a finite-dimensional split indecomposable weakly symmetric $k$-algebra with radical cube zero. 
In the notation of Notation \ref{notation}, we have $r=1$, $E=(\lambda)$, and $\lambda=\edim R$.
By \cite[Lemma 7.2(1)]{KT}, we have $\edim R\ge 3$.
Hence the assertion is a consequence of Theorem \ref{more than two}.
In fact, every finitely generated $R$-module $M$ with infinite projective dimension does not belong to $\add R$.
Theorem \ref{more than two} shows that $k\in \mod R=\thick_{\mod R}\{R,M\}$.

(2) Let $(R,\m,k)$ be such a commutative local ring. Since $R$ has an infinite residue field $k$, we can choose a system of parameters $x_1,\ldots, x_d$ of $R$ that generates a reduction of $\m$. Set $S=R/(x_1,\ldots, x_d)$ and let $\n$ be the maximal ideal of $S$. Then $S$ is an equicharacteristic Gorenstein local ring with $\ell_S(S)=e(R)\le \codim R+2=\edim S+2$. The equality 
$$
\ell_S(\n^2)=\ell_S(S)-(\ell_S(S/\n)+\ell_S(\n/\n^2))\le (\edim S+2)-(1+\edim S)=1
$$
shows either $\n^2=0$ or $\n^2\cong k$, and hence $\n^3=0$. Since $R$ is not a complete intersection, neither is $S$. By (1), $S$ is dominant, and therefore so is $R$ by \cite[Theorem 5.6]{T23}.
\end{proof}

Theorem \ref{main commutative}(1) answers the first case of Question \ref{ques dom Gor}(2) affirmatively in the equicharacteristic setting. 
Part (2) answers the second, higher-dimensional case when the residue field is infinite.
The result also gives a negative answer to Question \ref{ques dom Gor}(1).
Indeed, Example \ref{examples} exhibits Gorenstein dominant local rings that are not hypersurfaces, including the ring singled out in \cite[Remark 9.7]{T23}.
To the best of our knowledge, our result provides the first examples of Gorenstein dominant local rings that are not complete intersections, and in fact gives an entire class of such rings.

\begin{eg}\label{examples}
We provide explicit examples for parts (1) and (2) in Theorem \ref{main commutative}.
\begin{enumerate}[{\rm (1)}]
\item Let $k$ be a field and $e\ge 3$. Then 
$$k[x_1,\ldots,x_e]/(x_ix_j, x_k^2-x_1^2 \mid 1\le i<j\le e, \ 2\le k\le e)
$$
is a Gorenstein local ring that is not a complete intersection and whose maximal ideal has cube zero.
\item Let $k$ be an infinite field. The completions at their homogeneous maximal ideals of the following rings are equicharacteristic Gorenstein local rings with infinite residue fields that are not complete intersections and whose multiplicities are equal to their codimensions plus two.
\begin{enumerate}
\item The numerical semigroup ring $k[H]$ for a symmetric numerical semigroup $H$ such that
$$
m(H)=\operatorname{emb}(H)+1
\quad\text{and}\quad
\operatorname{emb}(H)\ge 4,
$$
where $m(H)$ and $\operatorname{emb}(H)$ denote the multiplicity and the embedding dimension of the numerical semigroup $H$, respectively; see \cite[Theorem]{K} for the Gorenstein property and \cite[Proposition 1(1) and (2)]{W} for the multiplicity and embedding dimension.
One concrete family of such numerical semigroup rings is
$$
k[t^{c+2},t^{c+3},\ldots,t^{2c+2}]\subset k[t] \quad \text{for}\ c\ge 3.
$$
\item The Stanley--Reisner rings of cycle graphs, viewed as simplicial triangulations of the circle:
$$
k[x_1,\ldots,x_{c+2}]/
(x_ix_j \mid 1\le i<j\le c+2,\ j-i\notin{\{1,c+1\}})\quad \text{for}\ c\ge 3;
$$
see \cite[Corollary 5.6.5]{BH} for the Gorenstein property
and \cite[Lemma 5.1.8 and Corollary 5.1.9]{BH} for the multiplicity. The embedding dimension is clearly $c+2$.
\end{enumerate}
\end{enumerate}
\end{eg}

Our results advance several lines of previous research from multiple perspectives.

\begin{rem}
(1) Dominance implies various homological properties of rings. For example, commutative rings satisfying the assertions of Corollary \ref{Ext Tor rigidity}(1) and (2) are called Ext-friendly and Tor-friendly, respectively.
By \cite[Theorem 6.7]{T23}, every dominant local ring is Ext/Tor-friendly, and satisfies the Auslander--Reiten conjecture. If $R$ is Gorenstein, then Tor-friendliness and Ext-friendliness further imply that it satisfies the homological vanishing conditions $(\mathbf{te})$, $(\mathbf{et})$, $(\mathbf{ee})$ and $(\mathbf{uac})$.
\begin{align*}
&\text{$(\mathbf{te})$ For $M,N\in\mod R$ with $\Tor_{\gg0}^R(M,N)=0$ one has $\Ext_R^{\gg0}(M,N)=0$.} \\
&\text{$(\mathbf{et})$ For $M,N\in\mod R$ with $\Ext_R^{\gg0}(M,N)=0$ one has $\Tor_{\gg0}^R(M,N)=0$.} \\
&\text{$(\mathbf{ee})$ For $M,N\in\mod R$ with $\Ext_R^{\gg0}(M,N)=0$ one has $\Ext_R^{\gg0}(N,M)=0$.} \\
&\text{$(\mathbf{uac})$ There is $n\ge 0$ such that $\Ext_R^{\gg0}(M,N)=0$ for $M,N\in\mod R$ implies $\Ext_R^{>n}(M,N)=0$.}
\end{align*}
A Gorenstein ring satisfying $(\mathbf{uac})$ is called an \textit{AB} ring.
For (Gorenstein) local rings satisfying either $\m^3=0$ or $e(R)\leq \codim R+2$, these properties were previously established from the viewpoint of rigidity by Huneke and Jorgensen; see \cite[Theorem 3.6 and Lemma 3.7]{HJ}. 
Our result recovers these vanishing results from the categorical property of dominance and strengthens them categorically.
Indeed, Ext/Tor-orthogonals naturally give rise to thick subcategories as in the proof of  Corollary \ref{Ext Tor rigidity}.

(2) In view of \cite[Lemma 7.2(1)]{KT}, Theorem \ref{main commutative}(1) amounts to imposing the three conditions that $(R,\m,k)$ is Gorenstein, that $\m^3=0$, and that $\edim R\ge 3$ on a commutative Artinian local ring. All three assumptions are essential: if any one of them is omitted, there exists a commutative Artinian local ring satisfying the other two that is not dominant.
Indeed, the Artinian local complete intersection $k[x,y]/(x^2,y^2)$ of embedding dimension two is Gorenstein and has cube-zero maximal ideal, but it is not a hypersurface, and hence is not dominant by \cite[Proposition 6.2(3)]{T23}.
Jorgensen and \c{S}ega \cite{JS} construct a standard graded Koszul Artinian Gorenstein local algebra $(R,\m)$ containing a field and satisfying $\m^4=0$, $\m^3\ne 0$, $\edim R=5$ but neither $(\mathbf{te})$ nor $(\mathbf{uac})$, and hence not dominant. In particular, the assumption $\m^3=0$ in Theorem \ref{main commutative} is optimal with respect to the nilpotency exponent.
In \cite[Remark 9.8]{T23}, a finite-dimensional algebra over a field with $\m^3=0$ and $\edim R=4$ is given that is neither Gorenstein nor dominant.
Recently, it was proved that this ring is not even quasi-dominant; see \cite[Theorem 5.7]{KMOT}.
\end{rem}

In the remainder of this paper, we make several observations from the viewpoint of commutative algebra concerning the cases $\lambda\le 2$ and the arguments used to prove them.
For commutative rings, the spectral radius $\lambda$ coincides with the embedding dimension and hence takes integer values. Thus, the case $\lambda<2$ occurs only for the field $k$ and the hypersurfaces $k[x]/(x^2)$ and $k[x]/(x^3)$, whose representation theory is straightforward. On the other hand, the case $\lambda=2$ corresponds to Artinian complete intersections of codimension two. These rings are not dominant and admit many nontrivial thick subcategories.
More generally, thick subcategories over complete intersections admit a geometric classification in terms of support varieties; see \cite{CI, KS, S} for instance. In the present case, the relevant support space is $\operatorname{Proj} k[x,y]\cong \mathbb{P}^1_k$, and the thick subcategories of $\underline{\operatorname{mod}}R$ are in an inclusion-preserving bijection with the specialization-closed subsets of $\mathbb{P}^1_k$.
Although we omit the details here to avoid redundancy, this correspondence can also be recovered from Theorem \ref{classification equal two} by linear-algebraic arguments.

The following example shows that a simple object need not be $1$-periodic. This distinction implies that a Serre subcategory of $\A$ need not be $\Omega$-stable.

\begin{eg}
Assume that $\operatorname{char}k\ne2$, and let $R=k[x,y]/(x^2,y^2)$ be a commutative Artinian complete intersection of codimension two. 
In the notation of Notation \ref{notation}, we have $r=1$ and $\lambda=2$.
Consider the $2$-periodic exact complex
$$
\cdots \xrightarrow{x-y} R \xrightarrow{x+y} R \xrightarrow{x-y} R \xrightarrow{x+y} \cdots.
$$
The cokernels of the differentials in this complex are alternately $S=R/(x-y)$ and $T=R/(x+y)$.
Since $\operatorname{char}k\ne2$, these two modules are not isomorphic. Indeed, their annihilator ideals are $(x-y)$ and $(x+y)$, respectively, and these ideals are distinct. Both $S$ and $T$ have length two. Since their only nonzero proper submodule is isomorphic to $k$, while $k$ does not have bounded Betti numbers, both $S$ and $T$ are simple objects of $\A_{\BB(R)}$.
Considering everything above in the graded setting, $\Omega$ interchanges $S$ and $T$. Hence the Serre subcategory of $\A$ generated by either $S$ or $T$ is not $\Omega$-stable.
\end{eg}

We conclude the paper by comparing the methods developed here with related arguments in the commutative algebra literature.
The use of gradings in Sections 2 and 3 is closely related to earlier work of Yoshino \cite{Y03} on totally reflexive modules over local rings with cube zero maximal ideal.  In the setting considered there, Yoshino showed that the ring admits a homogeneous grading $R=R_0\oplus R_1\oplus R_2$ with $R_0=k$ and $R_2=\mathfrak m^2$. He also showed that totally reflexive modules without free summands admit compatible gradings concentrated in two consecutive degrees and have linear free resolutions. The graded structures used in Sections 2 and 3 may be viewed as a representation-theoretic counterpart of this construction.

The numerical arguments used in Sections 3--5 also have close antecedents in the commutative algebra literature on short local rings.  Let $(R,\m,k)$ be a commutative local ring with $\m^3=0$ and $\soc R=\m^2$, and set $e=\dim_k\m/\m^2$ and $s=\dim_k\soc R$.  Lescot \cite{L} introduced exceptional modules, characterized by the absence of $k$ as a direct summand of their successive syzygies, and showed that their Betti numbers satisfy the recurrence
$$
\beta_{n+1}=e\beta_n-s\beta_{n-1}.
$$
In the Gorenstein case one has $s=1$, and this is precisely the scalar specialization of the recurrence
$$
\bfq_{n+1}=E\bfq_n-\bfq_{n-1}
$$
used in this paper.  In fact, for a proper thick subcategory $\X$, the objects of $\A_\X$ and all their iterated syzygies do not have $k$ as a direct summand, so that in the commutative local case they exhibit exactly the exceptional behavior considered in \cite{L}.

A related connection with Koszul modules is provided by Avramov, Iyengar, and \c{S}ega \cite{AIS}. They proved that, over a Gorenstein local ring with $\m^3=0$ and embedding dimension at least two, a finitely generated module is Koszul if and only if it has no direct summand isomorphic to $\Omega_R^{-i}k$ for any $i\ge1$. In particular, the indecomposable non-Koszul modules are precisely the negative syzygies of $k$. This characterization is also closely related to the properness of thick subcategories, and hence, in the commutative Artinian setting, to dominance.

\begin{ac}
The author would like to thank Ryo Takahashi and Yuya Otake for valuable comments.
The author was supported by JSPS Overseas Research Fellowships.
ChatGPT was used as a supplementary aid in the preparation of this manuscript.
\end{ac}

\end{document}